\documentclass[12]{amsart}
\usepackage{amsmath,amssymb,amsthm,color,enumerate,comment,centernot,enumitem,url,cite}
\usepackage{graphicx,relsize,bm}
\usepackage{mathtools}
\usepackage{array}

\makeatletter
\newcommand{\tpmod}[1]{{\@displayfalse\pmod{#1}}}
\makeatother

\newtheorem{thm}{Theorem}[section]
\newtheorem{lemma}[thm]{Lemma}

\newtheorem{cor}[thm]{Corollary}

\theoremstyle{remark}

\theoremstyle{definition}
    \newtheorem{defn}[thm]{Definition}

\newtheorem{rem}[thm]{Remark}

\theoremstyle{THM}

\newcommand{\abs}[1]{\left|{#1}\right|}

\def\FF {{\mathcal F}}
\def\GG {{\mathcal G}}

\def\Z {{\mathbb Z}}

\def\Q {{\mathbb Q}}

\def\C {{\mathcal C}}
\def\S {{{\mathcal S}}}

\def\GG {{\mathcal G}}

\def\F {{\mathbb F}}

\def\Z {{\mathbb Z}}
\def\Q {{\mathbb Q}}
\def\C {{\mathbb C}}
\def\S {{{\mathcal S}}}

\def\Gal{{\mbox{{\rm{Gal}}}}}

\makeatletter
\@namedef{subjclassname@2020}{%
  \textup{2020} Mathematics Subject Classification}
\makeatother

\def\red#1 {\textcolor{red}{#1 }}
\def\blue#1 {\textcolor{blue}{#1 }}

\numberwithin{equation}{section}

\def\Z {{\mathbb Z}}

\newcommand{\Mod}[1]{\ (\mathrm{mod}\enspace #1)}
\newcommand{\mmod}[1]{\ \mathrm{mod}\enspace #1}
\begin{document}

\title[Monogenic Octics and Their Galois Groups]{Monogenic Even Octic Polynomials\\ and Their Galois Groups}

%\author{Joshua Harrington}
%\address{Department of Mathematics, Cedar Crest College, Allentown, Pennsylvania, USA}
%\email[Joshua Harrington]{Joshua.Harrington@cedarcrest.edu}

\author{Lenny Jones}
\address{Professor Emeritus, Department of Mathematics, Shippensburg University, Shippensburg, Pennsylvania 17257, USA}
\email[Lenny~Jones]{doctorlennyjones@gmail.com}

%\author{Daniel White}
%\address{Department of Mathematics, Bryn Mawr College, Bryn Mawr, Pennsylvania 19010-2899, USA}
%\email[Daniel~White]{dfwhite@brynmawr.edu}
\date{\today}

\begin{abstract}
A monic polynomial $f(x)\in \Z[x]$ of degree $N$ is called \emph{monogenic} if $f(x)$ is irreducible over $\Q$ and $\{1,\theta,\theta^2,\ldots ,\theta^{N-1}\}$ is a basis for the ring of integers of $\Q(\theta)$, where $f(\theta)=0$. In a series of recent articles,  complete classifications of the Galois groups were given for irreducible polynomials
\[\FF(x):=x^8+ax^4+b\in \Z[x]\] and
\[\GG(x):=x^8+ax^6+bx^4+ax^2+1\in \Z[x], \quad a\ne 0.\]  In this article, for each Galois group $G$ arising in these classifications, we either construct an infinite family of monogenic octic polynomials $\FF(x)$ or $\GG(x)$ having Galois group $G$, or we prove that at most a finite such family exists. In the finite family situations, we determine all such polynomials. Here, a ``family" means that no two polynomials in the family generate isomorphic octic fields.
%
%We also provide a minor contribution to the existing partial classification of the Galois groups of $\GG(x)$ by giving simple conditions on the %coefficients $a$ and $b$ to determine when $\GG(x)$ is monogenic and the Galois group of $\GG(x)$ is $C_2\times D_4$ versus $C_2^2\wr C_2$.
\end{abstract}

% Secondary 11R32}

\maketitle
\section{Introduction}\label{Section:Intro}

 Unless stated otherwise, when we say that $f(x)\in \Z[x]$ is ``irreducible", we mean irreducible over $\Q$. We let $\Delta(f)$ and $\Delta(K)$ denote the discriminants over $\Q$, respectively, of $f(x)$ and a number field $K$.
If $f(x)$ is irreducible, with $f(\theta)=0$ and $K=\Q(\theta)$, then \cite{Cohen}
\begin{equation} \label{Eq:Dis-Dis}
\Delta(f)=\left[\Z_K:\Z[\theta]\right]^2\Delta(K),
\end{equation}
where $\Z_K$ is the ring of integers of $K$. We define a monic polynomial $f(x)\in \Z[x]$ to be \emph{monogenic} if $f(x)$ is irreducible and
$\Z_K=\Z[\theta]$, or equivalently from \eqref{Eq:Dis-Dis}, that  $\Delta(f)=\Delta(K)$. When $f(x)$ is monogenic, $\{1,\theta,\theta^2,\ldots ,\theta^{\deg(f)-1}\}$ is a basis for $\Z_K$, commonly referred to as a \emph{power basis}. The existence of a power basis facilitates computations in $\Z_K$, as in the case of the cyclotomic polynomials $\Phi_n(x)$ \cite{Washington}. A number field $K$ is defined to be \emph{mongenic} if there exists a power basis for $\Z_K$.

We caution the reader concerning two items. Certainly, the monogenicity of $f(x)$ implies the monogenicity of $K=\Q(\theta)$, where $f(\theta)=0$. However, the converse is not necessarily true. For example, let $f(x)=x^2-5$ and $K=\Q(\theta)$, where $\theta=\sqrt{5}$. Then, easy calculations show that
$\Delta(f)=20$ and $\Delta(K)=5$. Thus, $f(x)$ is not monogenic, but nevertheless, $K$ is monogenic since $\{1,(\theta+1)/2\}$ is a power basis for $\Z_K$.
%, and indeed $g(x)=x^2-x-1$ is monogenic .
A second item of concern is the following. We see from \eqref{Eq:Dis-Dis} that if $\Delta(f)$ is squarefree, then $f(x)$ is monogenic. However, the converse is false in general, and when $\Delta(f)$ is not squarefree, it can be quite difficult to determine whether $f(x)$ is monogenic.

 In a series of recent articles \cite{Awtrey1,Awtrey2,Chen,Awtrey3}, complete classifications of the Galois groups were given for irreducible polynomials \begin{equation}\label{Eq:F}
 \FF(x):=x^8+ax^4+b\in \Z[x],
 \end{equation} and
\begin{equation}\label{Eq:G}
\GG(x):=x^8+ax^6+bx^4+ax^2+1\in \Z[x],\quad a\ne 0.
\end{equation} These classifications provide $\Gal(\FF)$ and $\Gal(\GG)$, the Galois groups over $\Q$ of $\FF(x)$ and $\GG(x)$, respectively, by determining whether certain expressions involving only the coefficients of $\FF(x)$ and $\GG(x)$ are, or are not, squares in $\Z$. We point out that some of the results in \cite{Awtrey1,Awtrey2,Chen,Awtrey3} were given over $\Q$ or an arbitrary field of characteristic zero. However, in this article, we are only concerned with polynomials with coefficients in $\Z$.

 Using the standard ``8TX"-notation for transitive groups of degree 8 as given in MAGMA and \cite{BM, Dokchitser, KM}, the values of X that arise in the classifications in \cite{Awtrey1,Awtrey2,Chen,Awtrey3} for $\FF(x)$ and $\GG(x)$ are, respectively,
\begin{equation}\label{Eq:X}
X_{\FF}:=\{2,3,4,6,8,9,11,15,16,17,22,26\} \quad \mbox{and} \quad X_{\GG}:=\{2,3,4,9,10,18\}.
\end{equation}
Using $C_n$ to denote the cyclic group of order $n$, $D_n$ to denote the dihedral group of order $2n$, $Q_8$ to denote the quaterion group,  $\times$ to denote a direct product, $\rtimes$ to denote a semi-direct product, $.$ to denote a non-split extension, $\circ$ to denote a central product, $\wr$ to denote a wreath product and Hol to denote the holomorph, we provide in Table \ref{T:1} some more-familiar names for the groups in \eqref{Eq:X}.
    \begin{table}[h]
 \begin{center}
\begin{tabular}{c|ccccccc}
 X &  2 & 3 & 4 & 6 & 8 & 9 & 10  \\ [2pt]  %\hline \\[-8pt]
 8TX & $C_2\times C_4$ & $C_2^3$ & $D_4$ & $D_8$ & $Q_8\rtimes C_2$ & $C_2\times D_4$ & $C_2^2\rtimes C_4$  \\ [2pt] \hline\hline
 X &  11 & 15 & 16 & 17 & 18 & 22 & 26\\ [2pt]
 8TX  & $C_4\circ D_4$ & $C_8\rtimes C_2^2$ & $C_4.D_4$ & $C_4\wr C_2$ & $C_2^2\wr C_2$ & $D_4\circ D_4$ & Hol($D_4$)
 \end{tabular}
\end{center}
\caption{Familiar names for the groups 8TX in \eqref{Eq:X}}
 \label{T:1}
\end{table}

In this article, for each value of X in \eqref{Eq:X}, we %use the classifications in \cite{Awtrey1,Awtrey2,Chen}
either construct an infinite family of monogenic polynomials $\FF(x)$ or $\GG(x)$ in $\Z[x]$ having Galois group 8TX, or we prove that at most a finite such family exists. In the finite family situations, we determine all such polynomials. In certain cases, there are no such polynomials. By a ``family" we mean that all polynomials in the family are distinct in the sense that no two polynomials in the family generate isomorphic octic fields.

\begin{comment}
 If $\Gal(g)\simeq$ 4T3 $\simeq D_4$, where $g(x)=x^4+ax^3+bx^2+ax+1$, then in the existing partial classification of the Galois groups of $\GG(x)$ found in \cite{Chen}, no set of conditions on $a$ and $b$ is given to distinguish between $\Gal(\GG)\simeq$ 8T9 and $\Gal(\GG)\simeq$ 8T18 $\simeq C_2^2\wr C_2$. Deriving such conditions on $a$ and $b$ is difficult since it requires determining whether certain algebraic integers in $K=\Q(\theta)$ are perfect squares in $K$, where $\GG(\theta)=0$. However, since the square root of an algebraic integer is also an algebraic integer, we only need to show that these algebraic integers are perfect squares in $\Z_K$. Thus, our strategy here is to isolate infinite subsets of the coefficients $(a,b)$ so that $\GG(x)$ is monogenic. Then, $\Z_K=\Z[\theta]$ for these infinite subsets, and the computations become much more tractable. This approach allows us to provide a minor contribution to the classification of the Galois groups of $\GG(x)$ by giving simple conditions on the coefficients of $\GG(x)$ to make the distinction between the Galois groups 8T9 and 8T18 in the context of $\GG(x)$ being monogenic. %Furthermore, in either case of $\Gal(\GG)\in \{{\rm 8T9},{\rm 8T18}\}$, we prove, respectively, that there exist infinitely many prime values of $a$, or infinitely many prime values of both $a$ and $b$ such that $\GG(x)$ is monogenic.
\end{comment}

More precisely, we prove the following:
\begin{thm}\label{Thm:Main}
Let $\FF(x)\in \Z[x]$ and $\GG(x)\in \Z[x]$ be as defined in \eqref{Eq:F} and \eqref{Eq:G}. Let $X_{\FF}$ and $X_{\GG}$ be as defined in \eqref{Eq:X}.
 \begin{enumerate}
 \item \label{I:1} For each {\rm X} $\in\{9,15,17,26\}$, there exists at least one infinite family of monogenic polynomials $\FF(x)$ having Galois group {\rm 8TX}. For all other values of {\rm X} $\in X_{\FF}$, there exist at most finitely many monogenic polynomials $\FF(x)$ having Galois group {\rm 8TX}.
     \item \label{I:2} For each {\rm X} $\in \{9,18\}$, there exists at least one infinite family of monogenic polynomials $\GG(x)$ having Galois group {\rm 8TX}. For all other values of {\rm X} $\in X_{\GG}$, there exist at most finitely many monogenic polynomials $\GG(x)$ having Galois group {\rm 8TX}.
     %\item \label{I:2} For each {\rm X} $\in X_{\GG}$, there exist at most finitely many monogenic polynomials $\GG(x)$ having Galois group {\rm 8TX}.
     %\item \label{I:3}  If $\D:=(b+2-2a)(b+2+2a)(a^2-4b+8)$ is squarefree, and
     %\[(a \mmod{4},\ b \mmod{4})\in \{(1,3),(3,1),(3,3)\},\] %\Gamma$, where %(\widehat{a},\widehat{b})\in \Gamma$, where
     %\[\Gamma:=\{(1,3),(3,1),(3,3)\},\] % \quad \mbox{and} \quad \widehat{*}\in \{0,1,2,3\} \quad \mbox{ is the reduction modulo 4 of $*$},\]
     %then $\GG(x)$ is monogenic and
     %\[\Gal(\GG)\simeq \left\{ \begin{array}{cl}
      %{\rm 8T9} & \mbox{if $b=2a-1$}\\
      %{\rm 8T18} & \mbox{otherwise.}
     %\end{array}\right.\]  %oldssit f $\Gal(\GG)\simeq $ 8T9, then there exist infinitely many prime values of $a$, and if $\Gal(\GG)\simeq $ 8T18, then there exist infinitely many prime values of both $a$ and $b$ for which these conclusions hold.
\end{enumerate}
%Furthermore, there exist infinitely many integer pairs $(a,b)$ for which each case of item (3) of the theorem is valid.
\end{thm}
%The results of Theorem \ref{Thm:Main} are summarized in Tables \ref{T:2} and \ref{T:3} in Section \ref{Section:MainProof}.

% infinite families of
The literature involving the construction of monogenic or non-monogenic polynomials is fairly extensive
\cite{GSS,HJquartic,Jones1,JonesPAMS,JonesAJM,JonesIJM,JonesACM,JonesBAMS,SRM-Jones,JonesMS,JonesJAA,JonesAMI,JonesNYJM,JonesAJM23,JonesFACM23,JonesBAMS24,JP,K,Smith,Spearman,SWW}. However, the approach used in the proof of Theorem \ref{Thm:Main} differs from most previously-addressed situations in that each monogenic family is constructed with the specific goal that every polynomial in the family has the same Galois group. We do point out that a similar approach was carried out in \cite{HJquartic, JonesBAMS24} for monogenic quartic polynomials and their Galois groups. Of course, in the quartic case, a complete classification of the Galois groups is known \cite{KW}.

Another distinction of Theorem \ref{Thm:Main} is the fact that, for certain groups $G$, we also establish the nonexistence of a monogenic polynomial $\FF(x)$ or $\GG(x)$ having Galois group $G$.  Consequently, Theorem \ref{Thm:Main} is somewhat more in line with the following theorem of Gras \cite{Gras}:
\begin{thm}{\rm \cite{Gras}}\label{Thm:Gras}
  Let $\ell$ be a prime, and let $K$ be a degree-$\ell$ cyclic extension of $\Q$. If $\ell\ge 5$, then
$\Z_K$ does not have a power basis except in the case when $2\ell+1$ is prime and $K=\Q(\zeta_{2\ell+1}+\zeta^{-1}_{2\ell+1})$, the maximal real subfield of the cyclotomic field $\Q(\zeta_{2\ell+1})$, where $\Phi_{2\ell+1}(\zeta_{2\ell+1})=0$.
\end{thm}

By providing two examples, we illustrate how Theorem \ref{Thm:Gras} fits into the framework of Theorem \ref{Thm:Main}. Suppose first that $\ell=7$. Then, since $2\cdot 7+1=15$ is not prime, Theorem \ref{Thm:Gras} tells us that no monogenic polynomial of degree 7 exists having Galois group 7T1 $=C_7$. However, if $\ell=5$, then $2\cdot 5+1=11$ is prime, and we conclude from Theorem \ref{Thm:Gras} that there exists a degree-5 monogenic polynomial $f(x)$ having Galois group 5T1 $=C_5$. For example, if $f(x)=x^5+x^4-4x^3-3x^2+3x+1$, which is the minimal polynomial over $\Q$ of $\alpha=\zeta_{11}+\zeta^{-1}_{11}$, where $\zeta_{11}$ is a primitive $11$th root of unity, then $f(x)$ is monogenic and $\Gal(f)\simeq $ 5T1. Moreover, Theorem \ref{Thm:Gras} says that $f(x)$ is essentially the \emph{only} such polynomial, in the sense that if $g(x)$ is a degree-5 monogenic polynomial with $g(x)\ne f(x)$ and $\Gal(g)\simeq $ 5T1, then $\Q(\beta)=\Q(\alpha)$, where $g(\beta)=0$.

  % Computer computations in this article were done using either MAGMA, Maple or Sage.
\section{Preliminaries}\label{Section:Prelim}
When we refer to an expression as being a square, we mean a square in $\Z$.
The next theorem follows from \cite{Ljunggren}.
\begin{thm}\label{Thm:Ljunggren}
  Let $N\ge 0$ be an integer. Then $2^N-1$ is a square if and only if $N\in \{0,1\}$.
\end{thm}

The formula for the discriminant of an arbitrary monic trinomial, due to Swan \cite[Theorem 2]{Swan}, is given in the following theorem. %p. 1105
\begin{thm}
\label{Thm:Swan}
Let $f(x)=x^n+Ax^m+B\in \Q[x]$, where $0<m<n$, and let $d=\gcd(n,m)$. Then
\[
\Delta(f)=(-1)^{n(n-1)/2}B^{m-1}\left(n^{n/d}B^{(n-m)/d}-(-1)^{n/d}(n-m)^{(n-m)/d}m^{m/d}A^{n/d}\right)^d.
\]
\end{thm}

The next two theorems are transcriptions due to Awtrey and Patane \cite{Awtrey2} of the more-algorithmic classifications of $\FF(x)$ given in \cite{Chen}.
\begin{thm}\label{Thm:Chen1}
Let $\FF(x)$ be as defined in \eqref{Eq:F}. Suppose that $b$ and $\sqrt{b}$ are squares. Then $\Gal(\FF)$ is isomorphic to
\begin{enumerate}
  \item \label{C1:I1} {\rm 8T2} if and only if $-a^2+4b$ is a square,
  \item \label{C1:I2} {\rm 8T3} if and only if $a+2\sqrt{b}$ is a square,
  \item \label{C1:I3} {\rm 8T4} if and only if $a-2\sqrt{b}$ is a square,
  \item \label{C1:I4} {\rm 8T9} if and only if none of $-a^2+4b$, $a+2\sqrt{b}$ and $a-2\sqrt{b}$ is a square.
\end{enumerate}
\end{thm}

\begin{thm}\label{Thm:Chen2}
Let $\FF(x)$ be as defined in \eqref{Eq:F}. Suppose that $b$ is a square and $\sqrt{b}$ is not a square. Then $\Gal(\FF)$ is isomorphic to
\begin{enumerate}
  \item \label{C2:I1} {\rm 8T2} if and only if $a+2\sqrt{b}$ and $a\sqrt{b}-2b$ are both squares,
  \item \label{C2:I2} {\rm 8T4} if and only if both quantities in one of the following pairs are squares:
      \[(a+2\sqrt{b},-a\sqrt{b}+2b),\quad (a-2\sqrt{b},a\sqrt{b}+2b),\quad (-a\sqrt{b}+2b,a\sqrt{b}+2b).\]
  \item \label{C2:I3} {\rm 8T9} if and only if $(a-4b)\sqrt{b}$ is a not a square, and exactly one of
      \[a+2\sqrt{b},\quad a-2\sqrt{b},\quad -a\sqrt{b}+2b,\quad a\sqrt{b}+2b\]
      is a square.
  \item \label{C2:I4} {\rm 8T11} if and only if neither $-a\sqrt{b}+2b$ nor $a\sqrt{b}+2b$ is a square, and exactly one of
      \[-a^2+4b,\quad -a\sqrt{b}-2b,\quad (a^2-4b)\sqrt{b},\quad a\sqrt{b}-2b,\quad (4b-a^2)\sqrt{b}\] is a square,
  \item \label{C2:I5} {\rm 8T22} if and only if none of
  \[a+2\sqrt{b},\quad a-2\sqrt{b},\quad -a\sqrt{b}+2b,\quad a\sqrt{b}+2b,\]
  \[-a^2+4b,\quad -a\sqrt{b}-2b,\quad (a^2-4b)\sqrt{b},\quad a\sqrt{b}-2b,\quad (4b-a^2)\sqrt{b}\] is a square.
\end{enumerate}
\end{thm}
The next theorem is a compilation of the classifications given in \cite{Awtrey1} and \cite{Awtrey2}.
\begin{thm}\label{Thm:Awtrey} Let $\FF(x)$ be as defined in \eqref{Eq:F}.
Suppose that $b$ is not a square.
\begin{enumerate}
\item Then $\Gal(\FF)\simeq$ {\rm 8T16} if and only if $b(a^2-4b)$ is a square.
\item Then $\Gal(\FF)\simeq$ {\rm 8T17} if and only if $b(a^2-4b)$ is not a square and $4b-a^2$ is a square.
\item  Suppose further that neither $b(a^2-4b)$ nor $4b-a^2$ is a square.
    \begin{enumerate}
    \item Then $\Gal(\FF)\simeq$ {\rm 8T6} if and only if one of
\[2\sqrt{-b},\quad 4b+2\sqrt{-b(a^2-4b)}, \quad 4b-2\sqrt{-b(a^2-4b)}\]
is a nonzero square.
\item Then $\Gal(\FF)\simeq$ {\rm 8T8} if and only if %exactly one of
    \[2(a^2-4b)\sqrt{-b}\quad \mbox{or}\quad -4b+2\sqrt{-b(a^2-4b)}\]
    is a nonzero square. %Suppose further that $-b$ or $-b(a^2-4b)$ is a square.
\item Then $\Gal(\FF)\simeq$ {\rm 8T15} if and only if $-b$ or $-b(a^2-4b)$ is a square, and none of
\[2\sqrt{-b},\quad 2(a^2-4b)\sqrt{-b},\quad 4b+2\sqrt{-b(a^2-4b)},\] \[4b-2\sqrt{-b(a^2-4b)}, \quad -4b+2\sqrt{-b(a^2-4b)}\]
is a nonzero square.
\item Then $\Gal(\FF)\simeq$ {\rm 8T26} if and only if neither $-b$ nor $-b(a^2-4b)$ is a square.
    \end{enumerate}
\end{enumerate}
\end{thm}
The next two theorems present the classification of the Galois groups given in \cite{Awtrey3} %The next theorem is a compilation of the classifications given in \cite{Chen}
 for $\GG(x)$, as defined in \eqref{Eq:G}. We let
 \begin{equation}\label{Eq:Widef}
W_1:=b+2-2a,\quad W_2:=b+2+2a \quad \mbox{and} \quad W_3:=a^2-4b+8.
\end{equation}

%\begin{align}\label{Eq:Wdef}
%\begin{split}
%&W_1:=b+2-2a,\quad W_2:=b+2+2a, \quad W_3:=\frac{b-6+\sqrt{W_1W_2}}{2},\\
%&W_4:=\frac{b-6-\sqrt{W_1W_2}}{2} \quad \mbox{and} \quad W_5:=a^2-4b+8.
%\end{split}
%\end{align}

\begin{thm}\label{Thm:AP1}
  Let $\GG(x)$ be as defined in \eqref{Eq:G}. Then $\Gal(\GG)\simeq $
  \begin{enumerate}
       \item {\rm 8T2} if and only if exactly two of
       \[W_1W_3,\quad  W_2W_3 \quad \mbox{and} \quad W_1W_2W_3\quad \mbox{are squares,}\]
       \item  {\rm 8T3} if and only if all of
      \[W_1,\quad W_2 \quad \mbox{and} \quad W_1W_2 \quad \mbox{are squares,}\]
      \item  {\rm 8T10} if and only if exactly one of
      \[W_1W_3,\quad W_2W_3 \quad \mbox{and} \quad W_1W_2W_3 \quad \mbox{is a square,}\]
      \item  {\rm 8T18} if and only if none of
      \[W_1,\ W_2,\ W_1W_2,\ W_1W_3,\ W_2W_3 \ \mbox{ and }\ W_1W_2W_3 \ \mbox{ is a square.}\]
      \end{enumerate}
\end{thm}
%\begin{rem}\label{Rem:8T10}
%  Note that if $\Gal(f)\simeq$ 8T10, then we must have that $W_1\ne 1$ and $W_2\ne 1$ by Theorem \ref{Thm:AP1}.
%\end{rem}
\begin{thm}\label{Thm:AP2} $\Gal(\GG)\simeq $ {\rm 8T4} or {\rm 8T9} if and only if exactly one of
\[W_1,\quad W_2 \quad \mbox{and} \quad W_1W_2 \quad \mbox{is a square,}\] and none of
\[W_1W_3,\quad W_2W_3 \quad \mbox{and} \quad W_1W_2W_3 \quad \mbox{is a square.}\] Furthermore, if $\Gal(\GG)\simeq $ {\rm 8T4} or {\rm 8T9}, then
$\Gal(\GG)\simeq $ {\rm 8T4} if and only if\\
 either exactly one of
\[W_2\left(-a+4-2\sqrt{W_1}\right) \quad \mbox{and} \quad  W_2\left(-a+4+2\sqrt{W_1}\right) \quad \mbox{is a square}\]
\[\mbox{when $W_1$ is a square,}\]
or exactly one of
\[W_1\left(-a-4-2\sqrt{W_2}\right)  \quad \mbox{and} \quad  W_1\left(-a-4+2\sqrt{W_2}\right) \quad \mbox{is a square}\]
\[\mbox{when $W_2$ is a square,}\]
or exactly one of
\[W_2\left((12-2b-W_3)^2-4W_1W_2\right), \quad W_2\left(2b+W_3-12+2\sqrt{W_1W_2}\right)\]
\[\quad \mbox{and}\quad  W_2\left(2b+W_3-12-2\sqrt{W_1W_2}\right)\quad \mbox{is a square}\]
\[\mbox{when $W_1W_2$ is a square;}\]
 otherwise, $\Gal(\GG)\simeq $ {\rm 8T9}.
%\begin{align*}
%    \mbox{exactly one of}\quad  W_2\left(-a+4-2\sqrt{W_1}\right) & \quad \mbox{and} \quad  W_2\left(-a+4+2\sqrt{W_1}\right)\\
%      \quad \mbox{is a square} &  \mbox{ if $W_1$ is a square,}\\
%    \mbox{or exactly one of}\quad W_1\left(-a-4-2\sqrt{W_2}\right) & \quad \mbox{and} \quad  W_1\left(-a-4+2\sqrt{W_2}\right)\\
%     \quad \mbox{is a square} & \mbox{ if $W_2$ is a square,}\\
%    \mbox{or exactly one of}\quad W_2\left((12-2b-W_3)^2-4W_1W_2\right), & \quad W_2\left(2b+W_3-12+2\sqrt{W_1W_2}\right)\\
%     \quad \mbox{and}\quad  W_2\left(2b+W_3-12-2\sqrt{W_1W_2}\right) & \quad \mbox{is a square}\\
%      \mbox{if $W_1W_2$ is a square;}
%    \end{align*}
%   otherwise, $\Gal(\GG)\simeq $ {\rm 8T9}.
\end{thm}

\begin{comment}
\begin{thm}\label{Thm:Chen3}
  Let $\GG(x)$ be as defined in \eqref{Eq:G}.
  \begin{enumerate}
    \item Suppose that $W_1W_2$ is a square. Then
    \begin{enumerate}
      \item $\Gal(\GG)\simeq $ {\rm 8T2} if and only if $W_2$, $W_3$ and $W_4$ are not squares, and $W_2W_5$ is a square,
      \item $\Gal(\GG)\simeq $ {\rm 8T3} if and only if $W_2$ is a square,
      \item $\Gal(\GG)\simeq $ {\rm 8T4} if and only if $W_2$ is not a square and $W_3$ or $W_4$ is a square,
      \item $\Gal(\GG)\simeq $ {\rm 8T9} otherwise.
    \end{enumerate}
    \item Suppose that $W_1W_2$ is not a square and $W_1W_2W_5$ is a square. Then
    \begin{enumerate}
      \item $\Gal(\GG)\simeq $ {\rm 8T2} if and only if $W_1$ or $W_2$ is a square,
      \item $\Gal(\GG)\simeq $ {\rm 8T10} otherwise.
    \end{enumerate}
   \end{enumerate}
\end{thm}

The next two theorems are due to Capelli \cite{S}.
 \begin{thm}\label{Thm:Capelli1}{\rm \cite[Theorem 22]{S}}  Let $f(x)$ and $h(x)$ be polynomials in $\Q[x]$ with $f(x)$ irreducible. Suppose that $f(\alpha)=0$. Then $f(h(x))$ is reducible over $\Q$ if and only if $h(x)-\alpha$ is reducible over $\Q(\alpha)$.
 \end{thm}

\begin{thm}\label{Thm:Capelli2}{\rm \cite[Theorem 19]{S}}  Let $c\in \Z$ with $c\geq 2$, and let $\alpha\in\C$ be algebraic.  Then $x^c-\alpha$ is reducible over $\Q(\alpha)$ if and only if either there is a prime $p$ dividing $c$ such that $\alpha=\beta^p$ for some $\beta\in\Q(\alpha)$ or $4\mid c$ and $\alpha=-4\beta^4$ for some $\beta\in\Q(\alpha)$.
\end{thm}
\end{comment}

The following theorem, known as \emph{Dedekind's Index Criterion}, or simply \emph{Dedekind's Criterion} if the context is clear, is a standard tool used in determining the monogenicity of a polynomial.
\begin{thm}[Dedekind \cite{Cohen}]\label{Thm:Dedekind}
Let $K=\Q(\theta)$ be a number field, $T(x)\in \Z[x]$ the monic minimal polynomial of $\theta$, and $\Z_K$ the ring of integers of $K$. Let $q$ be a prime number and let $\overline{ * }$ denote reduction of $*$ modulo $q$ (in $\Z$, $\Z[x]$ or $\Z[\theta]$). Let
\[\overline{T}(x)=\prod_{i=1}^k\overline{\tau_i}(x)^{e_i}\]
be the factorization of $T(x)$ modulo $q$ in $\F_q[x]$, and set
\[h_1(x)=\prod_{i=1}^k\tau_i(x),\]
where the $\tau_i(x)\in \Z[x]$ are arbitrary monic lifts of the $\overline{\tau_i}(x)$. Let $h_2(x)\in \Z[x]$ be a monic lift of $\overline{T}(x)/\overline{h_1}(x)$ and set
\[F(x)=\dfrac{h_1(x)h_2(x)-T(x)}{q}\in \Z[x].\]
Then
\[\left[\Z_K:\Z[\theta]\right]\not \equiv 0 \pmod{q} \Longleftrightarrow \gcd\left(\overline{F},\overline{h_1},\overline{h_2}\right)=1 \mbox{ in } \F_q[x].\]
\end{thm}

The next result is the specific case for our octic situation of a ``streamlined" version of Dedekind's index criterion for trinomials that is due to Jakhar, Khanduja and Sangwan \cite{JKS2}.
\begin{thm}{\rm \cite{JKS2}}\label{Thm:JKS}
%Let $N\ge 2$ be an integer.
Let $K=\Q(\theta)$ be an algebraic number field with $\theta\in \Z_K$, the ring of integers of $K$, having minimal polynomial $\FF(x)=x^{8}+ax^4+b$ over $\Q$. A prime factor $q$ of $\Delta(\FF)=2^{16}b^3(a^2-4b)^4$ does not divide $\left[\Z_K:\Z[\theta]\right]$ if and only if $q$ satisfies on of the following conditions: %all of the following statements are true:
\begin{enumerate}[font=\normalfont]
  \item \label{JKS:I1} when $q\mid a$ and $q\mid b$, then $q^2\nmid b$;
  \item \label{JKS:I2} when $q\mid a$ and $q\nmid b$, then
  \[\mbox{either } \quad q\mid a_2 \mbox{ and } q\nmid b_1 \quad \mbox{ or } \quad q\nmid a_2\left(-ba_2^2-b_1^2\right),\]
  where $a_2=a/q$ and $b_1=\frac{b+(-b)^{q^j}}{q}$ with $q^j\mid\mid 8$;
  \item \label{JKS:I3} when $q\nmid a$ and $q\mid b$, then
  \[\mbox{either } \quad q\mid a_1 \mbox{ and } q\nmid b_2 \quad \mbox{ or } \quad q\nmid a_1b_2^3\left(-aa_1+b_2\right),\]
  where $a_1=\frac{a+(-a)^{q^e}}{q}$ with $q^e\mid\mid 4$, and $b_2=b/q$;
  \item \label{JKS:I4} when $q=2$ and $2\nmid ab$, then the polynomials
   \begin{equation*}
    H_1(x):=x^2+ax+b \quad \mbox{and}\quad H_2(x):=\dfrac{ax^4+b+\left(-ax-b\right)^4}{2}
   \end{equation*}
   are coprime modulo $2$;
         \item \label{JKS:I5} when $q\nmid 2ab$, then $q^2\nmid \left(a^2-4b\right)$.
    % \[q^2\nmid \left(4b-a^2\right).\]
   \end{enumerate}
\end{thm}

\begin{rem}
We will find both Theorem \ref{Thm:Dedekind} and Theorem \ref{Thm:JKS} useful in our investigations.
\end{rem}

\begin{comment}
The following theorem, known as \emph{Dedekind's Index Criterion}, or simply \emph{Dedekind's Criterion} if the context is clear, is a standard tool used in determining the monogenicity of a polynomial.
\begin{thm}[Dedekind \cite{Cohen}]\label{Thm:Dedekind}
Let $K=\Q(\theta)$ be a number field, $T(x)\in \Z[x]$ the monic minimal polynomial of $\theta$, and $\Z_K$ the ring of integers of $K$. Let $q$ be a prime number and let $\overline{ * }$ denote reduction of $*$ modulo $q$ (in $\Z$, $\Z[x]$ or $\Z[\theta]$). Let
\[\overline{T}(x)=\prod_{i=1}^k\overline{\tau_i}(x)^{e_i}\]
be the factorization of $T(x)$ modulo $q$ in $\F_q[x]$, and set
\[g(x)=\prod_{i=1}^k\tau_i(x),\]
where the $\tau_i(x)\in \Z[x]$ are arbitrary monic lifts of the $\overline{\tau_i}(x)$. Let $h(x)\in \Z[x]$ be a monic lift of $\overline{T}(x)/\overline{g}(x)$ and set
\[F(x)=\dfrac{g(x)h(x)-T(x)}{q}\in \Z[x].\]
Then
\[\left[\Z_K:\Z[\theta]\right]\not \equiv 0 \pmod{q} \Longleftrightarrow \gcd\left(\overline{F},\overline{g},\overline{h}\right)=1 \mbox{ in } \F_q[x].\]
\end{thm}
\end{comment}

\begin{comment}
The next theorem follows from Corollary (2.10) in \cite{Neukirch}.
\begin{thm}\label{Thm:CD}%{\rm \cite{Neukirch}}
  Let $K$ and $L$ be number fields with $K\subset L$. Then \[\Delta(K)^{[L:K]} \bigm|\Delta(L).\]
\end{thm}
\end{comment}

\begin{thm}\label{Thm:Pasten}
   Let $G(t)\in \Z[t]$, and suppose that $G(t)$ factors into a product of distinct non-constant polynomials $\gamma_i(t)\in \Z[x]$ that are irreducible over $\Z$, such that the degree of each $\gamma_i(t)$ is at most 3.   Define %\left\vert\vphantom{\frac{1}{1}}\right.
   %\begin{equation}\label{Eq:NG}
   \[N_G\left(X\right)=\abs{\left\{p\le X : p \mbox{ is prime and } G(p) \mbox{ is squarefree}\right\}}.\]
   %\end{equation}
   Then, %the following asymptotic holds unconditionally if $d\le 3$, and holds, assuming the abc-conjecture for number fields for $f(x)$, if $d\ge 4$:
   \begin{equation}\label{Eq:NG}
   N_G(X)\sim C_G\dfrac{X}{\log(X)},
   \end{equation}
   where
   \begin{equation}\label{Eq:CG}
   C_G=\prod_{\ell  \mbox{ \rm {\tiny prime}}}\left(1-\dfrac{\rho_G\left(\ell^2\right)}{\ell(\ell-1)}\right)
   \end{equation}
   and $\rho_G\left(\ell^2\right)$ is the number of $z\in \left(\Z/\ell^2\Z\right)^{*}$ such that $G(z)\equiv 0 \pmod{\ell^2}$.
 \end{thm}

\begin{rem}
  Theorem \ref{Thm:Pasten} follows from work of Helfgott, Hooley and Pasten \cite{Helfgott-cubic,Hooley-book,Pasten}. For more details, see the discussion following \cite[Theorem 2.11]{JonesNYJM}.
\end{rem}

\begin{defn}\label{Def:Obstruction}
 In the context of Theorem \ref{Thm:Pasten}, for $G(t)\in \Z[t]$ and a prime $\ell$, if $G(z)\equiv 0 \pmod{\ell^2}$ for all $z\in \left(\Z/\ell^2\Z\right)^{*}$, we say that $G(t)$ has a \emph{local obstruction at $\ell$}. A polynomial $G(t)\in \Z[t]$ is said to have \emph{no local obstructions}, if for every prime $\ell$ there exists some $z\in \left(\Z/\ell^2\Z\right)^{*}$ such that $G(z)\not \equiv 0 \pmod{\ell^2}$.
\end{defn}

Note that $C_G>0$ in \eqref{Eq:CG} if and only if $G(t)$ has no local obstructions. Consequently, it follows that $N_G(X)\to \infty$ as $X\to \infty$ in \eqref{Eq:NG}, when $G(t)$ has no local obstructions. Hence, we have the
following immediate corollary of Theorem \ref{Thm:Pasten}.
\begin{cor}\label{Cor:Squarefree}
 Let $G(t)\in \Z[t]$, and suppose that $G(t)$ factors into a product of distinct non-constant polynomials $\gamma_i(t)\in \Z[x]$ that are irreducible over $\Z$, such that the degree of each $\gamma_i(t)$ is at most 3. To avoid the situation when $C_G=0$ (in \eqref{Eq:CG}), we suppose further that $G(t)$ has no local obstructions.
  Then there exist infinitely many primes $p$ such that $G(p)$ is squarefree.
\end{cor}
The following lemma, which generalizes a discussion found in \cite{JonesBAMS}, will be useful in the proof of Theorem \ref{Thm:Main}.
\begin{lemma}\label{Lem:ObstructionCheck}
  Let $G(t)\in \Z[t]$ with $\deg(G)=N$, and suppose that $G(t)$ factors into a product of distinct non-constant polynomials that are irreducible over $\Z$, such that the degree of each factor is at most 3. If $G(t)$ has an obstruction at the prime $\ell$, then $\ell\le (N_{\ell}+2)/2$, where $N_{\ell}$ is the number of not-necessarily distinct non-constant linear factors of $G(t)$ in $\F_{\ell}[t]$.
\end{lemma}
\begin{proof}
Since no factors of $G(t)$ in $\Z[t]$ are constant, we can assume that the content of every factor of $G(t)$ is 1. %In particular, if $rt+s$ is a factor of $G(t)$, then $\gcd(r,s)=1$.
Furthermore, since a nonlinear factor of $G(t) \Mod{\ell}$ never has a zero in $\left(\Z/\ell^2\Z\right)^*$, we can also assume, without loss of generality, that $G(t)$ factors completely into $N$, not-necessarily distinct, non-constant linear factors in $\Z[t]$. Thus,
\begin{equation}\label{Eq:G(t)}
 G(t)\equiv c\prod_{j=0}^{\ell-1}(t-j)^{e_j} \pmod{\ell},
 \end{equation} where $c\not \equiv 0 \pmod{\ell}$, $e_j\ge 0$ for each $j$ and $N=\sum_{j=0}^{\ell-1}e_j$.
   Observe that if $e_j=0$ for some $j\ne 0$ in \eqref{Eq:G(t)}, then $G(j)\not \equiv 0 \Mod{\ell^2}$, contradicting the fact that $G(t)$ has an obstruction at the prime $\ell$. If $e_j=1$ for some $j\ne 0$ in \eqref{Eq:G(t)}, then the zero $j$ of $x-j \Mod{\ell}$ lifts to the unique zero $j$ of $x-j \Mod{\ell^2}$. Thus, $G(j+\ell)\not \equiv 0 \Mod{\ell^2}$, again contradicting the fact that $G(t)$ has an obstruction at the prime $\ell$. Hence, $e_j\ge 2$ for all $j\in \{1,2,\ldots,\ell-1\}$. Assume, by way of contradiction, that $\ell> (N+2)/2$. Then
  \[2(\ell-1)>N=\sum_{j=0}^{\ell-1}e_j=e_0+\sum_{j=1}^{\ell-1}e_j\ge e_0+2(\ell-1),\] which is impossible, and the proof is complete.
\end{proof}

%\begin{comment}

%\end{comment}

\section{The Proof of Theorem \ref{Thm:Main}}\label{Section:MainProof}
  In certain situations of the proof of item \eqref{I:2} of Theorem \ref{Thm:Main},  it will be convenient to examine the monogenicity of
\begin{equation}\label{Eq:g}
g(x):=x^4+ax^3+bx^2+ax+1,
\end{equation}
 in light of the fact that
\begin{equation}\label{Eq:gnmGnm}
\mbox{if $g(x)$ is not monogenic, then $\GG(x)=g(x^2)$ is not monogenic.}
\end{equation}
%and we use Theorem \ref{Thm:Dedekind} to do so.
Throughout this section, we let $\GG(x)$ be as defined in \eqref{Eq:G}, and we let $W_1$, $W_2$ and $W_3$ be as defined in \eqref{Eq:Widef}.
Straightforward computations reveal that
\[\Delta(g)=(b+2-2a)(b+2+2a)(a^2-4b+8)^2=W_1W_2W_3^2 \quad \mbox{and}\quad \Delta(\GG)=2^8\Delta(g)^2.\]

Before we present the proof of Theorem \ref{Thm:Main}, we prove some lemmas that will be useful for the proof of item \eqref{I:2} of Theorem \ref{Thm:Main}. The first lemma follows from \cite[Theorem 1.1]{JonesBAMS}.
 \begin{lemma}\label{Lem:BAMS} If
    \[(a \mmod{4}, \ b \mmod{4})\in \{(1,3),(3,1),(3,3)\},\] and $W_1W_2W_3$ is squarefree, then $\GG(x)$ is monogenic.
 \end{lemma}

 \begin{lemma}\label{Lem:MainG2}
\text{}
\begin{enumerate}
  \item \label{LMG2:1} If there exists a prime $q$, such that $q^2\mid W_1$ or $q^2\mid W_2$, then $\GG(x)$ is not monogenic.
  \item \label{LMG2:2} If $W_1$ and $W_2$ are squarefree and there exists a prime $q\ge 3$ such that $q^2\mid W_3$, then $\GG(x)$ is not monogenic.
  \item \label{LMG2:3} Suppose that $W_1$ and $W_2$ are squarefree, and that $W_3$ is not divisible by the square of an odd prime. If $(a \mmod{4}, \ b\mmod{4})\in \{(0,1), (2,3)\}$, then $\GG(x)$ is not monogenic.
\end{enumerate}
\end{lemma}
 \begin{proof}
   For all items of the lemma, it is enough, by \eqref{Eq:gnmGnm}, to show that $g(x)$, as defined in \eqref{Eq:g}, is not monogenic.

   We begin with item \eqref{LMG2:1}, and we assume that $q^2\mid W_1$. We present details only in this case since the case $q^2\mid W_2$ is similar. Because $q^2\mid W_1$, it follows that $b\equiv 2a-2\pmod{q}$ and
\[g(x)\equiv x^4+ax^3+(2a-2)x^2+ax+1\equiv (x+1)^2g_1(x) \pmod{q},\]
where $g_1(x)=x^2+(a-2)x+1$ with $\Delta(g_1)=a(a-4)$. The three possibilities for the quadratic polynomial $g_1(x)$ over $\F_q$ are:
%\begin{subequations}
\begin{align}%\label{Eq:3poss}
g_1(x)& \mbox{ is irreducible}, \label{P1}\\
g_1(x)& \mbox{ has a double zero}, \label{P2}\\
g_1(x)&  \mbox{ has two distinct zeros.} \label{P3}
\end{align}
%\end{subequations}
  We use Theorem \ref{Thm:Dedekind} with the prime $q$ and $T(x):=g(x)$ to show in each of these possibilities that $x+1$ divides $\gcd(\overline{F},\overline{h_1},\overline{h_2})$.

For possibility \eqref{P1}, we can let \[h_1(x)=(x+1)g_1(x) \quad \mbox{and}\quad h_2(x)=x+1.\]
Then \[qF(x)=h_1(x)h_2(x)-T(x)=(x+1)^2g_1(x)-g(x)=-(b+2-2a)x^2\equiv 0 \pmod{q^2},\]
from which we conclude that $g(x)$ is not monogenic.

For possibility \eqref{P2}, we have that
\[g_1(x)\equiv \left\{\begin{array}{cl}
   (x+1)^2 \pmod{q} & \mbox{if and only if $q\mid (a-4)$}\\
  (x-1)^2 \pmod{q} & \mbox{if and only if $q\mid a$.}
\end{array}\right.\] Thus, we can choose $h_1(x)$ and $h_2(x)$ so that
\[h_1(x)h_2(x)=\left\{\begin{array}{cl}
  (x+1)^4 & \mbox{if and only if $q\mid (a-4)$,}\\
  (x+1)^2(x-1)^2 & \mbox{if and only if $q\mid a$,}
\end{array}\right.\] and
 \[qF(x)= \left\{\begin{array}{cl}
  -((a-4)x^2+(b-6)x+a-4) & \mbox{if and only if $q\mid (a-4)$}\\
  -(ax^2+(b+2)x+a) & \mbox{if and only if $q\mid a$.}
\end{array}\right.\]
Since, for both of these cases, we have that
\[qF(-1)=b+2-2a\equiv 0 \pmod{q^2},\]
it follows that $g(x)$ is not monogenic for this possibility.

Finally, for possibility \eqref{P3}, suppose that $g_1(x)\equiv (x-c)(x-d) \pmod{q}$, with $c\not \equiv d \pmod{q}$. Note that
\[g_1(-1)=(-1)^2+(a-2)(-1)+1=-(a-4)\not \equiv 0 \pmod{q},\]
so that $c\not \equiv -1 \pmod{q}$ and $d\not \equiv -1 \pmod{q}$.  Then, we can let
\[h_1(x)=(x+1)(x-c)(x-d)\quad \mbox{and} \quad h_2(x)=x+1.\] Then,
\[h_1(x)h_2(x)=(x+1)^2(x-c)(x-d),\] so that
\begin{align*}
qF(x)&=(x+1)^2(x-c)(x-d)-g(x)\\
&=(-c+2-a-d)x^3+(-2c+1+dc-2d-b)x^2\\
&\qquad \qquad +(-d+2dc-a-c)x+dc-1.
\end{align*}
Observe then that
\[qF(-1)=-(b+2-2a)\equiv 0 \pmod{q^2}.\]
Hence, $g(x)$ is not monogenic in this possibility as well, and the proof of item \eqref{LMG2:1} is complete.

We turn next to item \eqref{LMG2:2}, and we let $q\ge 3$ be a prime divisor of $W_3=a^2-4b+8$ such that $q^2\mid W_3$. Then $b\equiv (a^2+8)/4 \pmod{q}$ and
    \[g(x)\equiv (x^2+(a/2)x+1)^2 \pmod{q}.\]  We consider the three possibilities for the polynomial
    \begin{equation}\label{Eq:g1}
    g_1(x):=x^2+(a/2)x+1\in \F_q[x]:
    \end{equation}
  \begin{enumerate}
  \item \label{Poss1} $g_1(x)$ is irreducible,
  \item \label{Poss2} $g_1(x)$ has a double zero,
  \item \label{Poss3} $g_1(x)$ has two distinct zeros.
\end{enumerate}
 We use Theorem \ref{Thm:Dedekind} with the prime $q$ and $T(x):=g(x)$ to examine each of these possibilities.

If $g_1(x)$ is irreducible in $\F_q[x]$, then we can let
\[h_1(x)=h_2(x)=x^2+\left(\dfrac{q^2+1}{2}\right)ax+1\in \Z[x],\] so that
\begin{align*}
qF(x)&=h_1(x)h_2(x)-T(x)\\
&=\left(x^2+\left(\dfrac{q^2+1}{2}\right)ax+1\right)^2-g(x)\\
&=-x\left(-aq^2x^2+\left(b-2-\dfrac{(q+1)^2}{4}a^2\right)x-aq^2\right)\\
&\equiv \left(\dfrac{a^2-4b+8}{4}\right)x^2 \pmod{q^2}\\
&\equiv 0 \pmod{q^2}.
\end{align*}
Hence, $g(x)$ is not monogenic.

Suppose next that $g_1(x)$ has a double zero in $\F_q[x]$. Then $q\mid (a^2-16)$, since $\Delta(g_1)=(a^2-16)/4$. Thus, $q\mid(b-6)$ since
$a^2-16-(a^2-4b+8)=4(b-6)$ and $q\ge 3$. Consequently,
 \[W_1W_2=(b+2)^2-4a^2\equiv (b+2)^2-4(4b-b)\equiv (b-6)^2 \equiv 0 \pmod{q^2}.\] Hence, since $W_1$ and $W_2$ are squarefree, it follows that $q\mid \gcd(W_1,W_2)$. Then $q\mid a$, since $q$ divides $W_1-W_2=-4a$, and $q\ge 3$. Thus, $q$ divides $a^2-(a^2-16)=16$, which contradicts the fact that $q\ge 3$. Therefore, $g_1(x)$ never has a double zero in $\F_q[x]$ when $q\ge 3$.

 Next, with $g_1(x)$ as defined in \eqref{Eq:g1}, suppose that $g_1(x)\equiv (x-c)(x-d) \pmod{q}$, where $c\not \equiv d \pmod{q}$. Then

 \[x^2+\left(\dfrac{q^2+1}{2}\right)ax+1=(x-c)(x-d)-qr(x),\] for some linear polynomial $r(x)\in \Z[x]$. Thus, with $T(x)=g(x)$ and
 \[h_1(x)=h_2(x)=(x-c)(x-d)=x^2+\left(\dfrac{q^2+1}{2}\right)ax+1+qr(x),\] a straightforward computation reveals that
 \begin{align*}
 F(x)&=\dfrac{h_1(x)h_2(x)-T(x)}{q}\\
   &=aqx^3+\left(\dfrac{a^2-4b+8}{q}+\dfrac{a^2(q^3+2q)}{4}+2r(x)\right)x^2\\
   &\qquad \qquad \qquad +(aq^2r(x)+aq+ar(x))x+qr(x)^2+2r(x)\\
   &\equiv 2r(x)\left(x^2+\left(\dfrac{a}{2}\right)x+1\right) \pmod{q}\\
   &\equiv 2r(x)(x-c)(x-d) \pmod{q},
 \end{align*}
 which implies that $\gcd(\overline{F},\overline{h_1})\ne 1$. Hence, again, $g(x)$ is not monogenic, and the proof of item \eqref{LMG2:2} is complete.

 Finally, we consider item \eqref{LMG2:3}. Since $2\mid a$, it follows that $2\mid W_3$, so that $2\mid \Delta(g)$. Then, we apply Theorem \ref{Thm:Dedekind} with $T(x):=g(x)$ and $q=2$. Since $(a \mmod{4}, \ b\mmod{4})\in \{(0,1), (2,3)\}$, we have that $\overline{T}(x)=(x^2+x+1)^2$. Therefore, we can let $h_1(x)=h_2(x)=x^2+x+1$ since $x^2+x+1$ is irreducible in $\F_2[x]$.  Hence,
  \begin{align*}
 F(x)&=\dfrac{h_1(x)h_2(x)-T(x)}{q}\\
 &=\dfrac{(x^2+x+1)^2-g(x)}{2}\\
 &=-x\left(\left(\frac{a}{2}\right)x^2+\left(\frac{b-1}{2}\right)x+\frac{a}{2}\right)\\
 &\equiv \left\{\begin{array}{cl}
   0 \pmod {2}& \mbox{if $(a \mmod{4}, \ b\mmod{4})=(0,1)$}\\[.5em]
   x(x^2+x+1) \pmod{2} & \mbox{if $(a \mmod{4}, \ b\mmod{4})=(2,3)$.}
 \end{array}\right.
 \end{align*}
 Thus, we see in either case, that $\gcd(\overline{F},\overline{h_1})\ne 1$ so that $g(x)$ is not monogenic, and the proof of the lemma is complete.
\end{proof}

%\begin{lemma}\label{Lem:MainG3}
%  Let $K=\Q(\theta)$ and $\Z_K$ denote the ring of integers of $K$, where $\GG(\theta)=0$. Suppose that $W_1$ and $W_2$ are squarefree with $2\nmid W_1W_2$, %and that $4\mid W_3$. Then $[\Z_K:\Z[\theta]]\equiv 0 \pmod{2}$ if and only if $(a \mmod{4}, \ b\mmod{4})\in \{(0,1),(2,3)\}$.
%\end{lemma}
\begin{proof}[Proof of Theorem \ref{Thm:Main}]
We follow closely the classifications given in Theorems \ref{Thm:Chen1}, \ref{Thm:Chen2}, \ref{Thm:Awtrey}, \ref{Thm:AP1} and \ref{Thm:AP2}. To determine the monogenic polynomials, we use Theorem \ref{Thm:JKS} for the proof of item \eqref{I:1}, and Theorem \ref{Thm:Dedekind} for the proof of item \eqref{I:2}. %  to determine the monogenic polynomials.

We begin with the proof of item \eqref{I:1}.
From Theorem \ref{Thm:Swan}, we have that
\begin{equation}\label{Eq:DelF}
\Delta(\FF)=2^{16}b^3(a^2-4b)^4.
\end{equation}
We see from the classification for $\FF(x)$ given in Theorem \ref{Thm:Chen1}, Theorem \ref{Thm:Chen2} and Theorem \ref{Thm:Awtrey} %\cite{Awtrey1,Chen,Awtrey2}
 that $X_{\FF}$ in \eqref{Eq:X} can be partitioned into the two sets
\begin{equation}\label{Eq:Partition}
X_{\FF}^{\square}=\{2,3,4,9,11,22\} \quad \mbox{and} \quad X_{\FF}^{\not\square}=\{6,8,15,16,17,26\},
\end{equation}
where each element in $X_{\FF}^{\square}$ arises from polynomials $\FF(x)$ with $b$ a square, while each element in $X_{\FF}^{\not\square}$ arises from polynomials $\FF(x)$ with $b$ not a square. This partitioning is useful for the following reason. Suppose that $b$ is a square, and $q$ is a prime that divides $b$. We see then that condition \eqref{JKS:I1} of Theorem \ref{Thm:JKS} is false if $q\mid a$, while condition \eqref{JKS:I3} of Theorem \ref{Thm:JKS} is false if $q\nmid a$. Thus, the only possible monogenic polynomials in these cases must have $b=1$. Consequently, we have immediately that the situations in the classification where $b$ is a square and $\sqrt{b}$ is not a square (see Theorem \ref{Thm:Chen2}) have no monogenic polynomials $\FF(x)$.
%Thus, there are no monogenic polynomials arising from the conditions in Theorem \ref{Thm:Chen2}, and
In particular, there are no monogenic polynomials $\FF(x)$ having Galois group 8T11 or 8T22.

%We begin with the proof of item \eqref{I:1} and
We systematically proceed through the values of X in $X_{\FF}\setminus \{11,22\}$, beginning with the values of X in $X_{\FF}^{\square}\setminus \{11,22\}$ given in \eqref{Eq:Partition}. %From \eqref{Eq:DelF}, we see that
%\[\Delta(\FF)=2^{16}(a^2-4)^4.\]
\subsection{Values of X in $X_{\FF}^{\square}\setminus \{11,22\}$}

We assume that $b=1$, and we use Theorem \ref{Thm:Chen1}.
\subsection*{X=2} We assume that $\Gal(\FF)\simeq$ 8T2.
Then, from Theorem \ref{Thm:Chen1}, we have that $-a^2+4$ is a square. If $q$ is an odd prime prime dividing $-a^2+4$, we have that $q\nmid 2a$, and we conclude that condition \eqref{JKS:I5} of Theorem \ref{Thm:JKS} is false since $-a^2+4$ is a square. Thus, we must have that $-a^2+4=2^{2m}$, for some integer $m\ge 0$. Therefore,
\[4-2^{2m}=a^2\ge 0,\]
which implies that
 $m=1$ and $a=0$. Hence, the only monogenic polynomial is this case is the cyclotomic polynomial $\Phi_{16}(x)=x^8+1$.
\subsection*{X=3} We assume that $\Gal(\FF)\simeq$ 8T3. By Theorem \ref{Thm:Chen1}, we have that $a+2$ is a square.
Suppose that $q$ is an odd prime dividing $a+2$. Then $q\nmid 2a$, and since $a+2$ is a square, we deduce that $q^2\mid (a^2-4)$, making condition \eqref{JKS:I5} of Theorem \ref{Thm:JKS} false.

Hence, $a+2=2^{2m}$, for some integer $m\ge 0$.
If $m\ge 1$, then $2\mid a$, but $4 \nmid a$. Thus, $2\nmid a_2$ and $b_1=1$. Since
\[a_2\left(a_2^2-(-1)^2\right)\equiv 0 \pmod{2},\] we conclude that condition \eqref{JKS:I2} of Theorem \ref{Thm:JKS} is false.
If $m=0$, then $a=-1$, and it is easy to check that $\FF(x)=x^8-x^4+1$ is indeed monogenic. Therefore, there is exactly one monogenic polynomial in this situation.
\subsection*{X=4} We assume that $\Gal(\FF)\simeq$ 8T4. Using Theorem \ref{Thm:Chen1}, the approach  is similar to X=3, and so we omit the details. The only monogenic in this case is $\FF(x)=x^8+3x^4+1$.
\subsection*{X=9} We employ a slightly different strategy here. Let $G(t)=(4t+1)(4t+5)$. We claim that $G(t)$ has no obstructions. By Lemma \ref{Lem:ObstructionCheck}, we only have to check for obstructions at the prime $\ell=2$, which is easily confirmed since $G(1)\not \equiv 0 \pmod{4}$. Hence, by Corollary \ref{Cor:Squarefree}, there exist infinitely many primes $p$ such that $G(p)$ is squarefree. Let $p$ be such a prime. We claim that $\FF(x)=x^8+(4p+3)x^4+1$ is monogenic with $\Gal(\FF)\simeq $ XT9. We see from \cite[Lemma 3.1]{JonesBAMS} that $\FF(x)$ is irreducible.  Since $(4p+1)(4p+5)$ is squarefree, it follows from Theorem \ref{Thm:Chen1} that $\Gal(\FF)\simeq $ 8T9. We use Theorem \ref{Thm:JKS} to verify that $\FF(x)$ is monogenic. Note that
\[\Delta(\FF)=2^{16}(4p+1)^4(4p+5)^4\] from \eqref{Eq:DelF}. If $q$ is an odd prime dividing $\Delta(\FF)$, then $q\nmid 2(4p+3)$; and every condition of Theorem \ref{Thm:JKS} is true, including condition \eqref{JKS:I5} since $(4p+1)(4p+5)$ is squarefree. If $q=2$, then $2\nmid (4p+3)$ and we need to check condition \eqref{JKS:I4} of Theorem \ref{Thm:JKS}. An easy calculation reveals that
\[ H_2(x)= \dfrac{(4p+3)x^4+1+(-(4p+3)x-1)^4}{2}\equiv (x+1)^2 \pmod{2}. \]
  Since $H_1(x)=x^2+(4p+3)x+1\equiv x^2+x+1 \pmod{2}$ is irreducible in $\F_2[x]$, we deduce that condition \eqref{JKS:I4} of Theorem \ref{Thm:JKS} is true.
 Thus, $\FF(x)$ is monogenic. By equating discriminants, Maple confirms that there does not exist a prime $p^{\prime}$ such that the octic field generated by $x^8+(4p'+3)x^4+1$ is isomorphic to the octic field generated by $\FF(x)=x^8+(4p+3)x^4+1$. Thus, the set of all such $\FF(x)$, where $(4p+1)(4p+5)$ is squarefree does indeed represent an infinite family of distinct monogenic 8T9-octic trinomials.
\begin{rem}
  We point out to the reader that the family of monogenic polynomials described above does not constitute all monogenic polynomials $\FF(x)=x^8+ax^4+b$ with $\Gal(\FF)\simeq$ 8T9.
\end{rem}
 \subsection{Values of X in $X_{\FF}^{\not\square}$}
 We turn now to  values of X in $X_{\FF}^{\not\square}$ given in \eqref{Eq:Partition}. We assume that $b$ is not a square, and we use Theorem \ref{Thm:Awtrey}. \subsection*{X=6} Following Theorem \ref{Thm:Awtrey}, we suppose that neither $b(a^2-4b)$ nor $4b-a^2$ is a square, and that
%\[b,\quad b(a^2-4b),\quad 4b-a^2\] is a square and that
$\Gal(\FF)\simeq$ 8T6. Then one of
\begin{equation}\label{Eq:possiblesquares}
2\sqrt{-b},\quad 4b+2\sqrt{-b(a^2-4b)}, \quad 4b-2\sqrt{-b(a^2-4b)}
\end{equation}
is a nonzero square.

We consider the three possibilities in \eqref{Eq:possiblesquares} one at a time, but first we make the following useful observation.
Suppose that $q$ is a prime such that $q^2\mid b$. If $q\mid a$, then condition \eqref{JKS:I1} of Theorem \ref{Thm:JKS} is false. If $q\nmid a$, then $b_2=b/q\equiv 0 \pmod{q}$, which implies that condition \eqref{JKS:I3} of Theorem \ref{Thm:JKS} is false. Hence, if $\FF(x)$ is monogenic, then $b$ must be squarefree. %being squarefree is a necessary condition for the monogenicity of $\FF(x)$.
Thus, since $-b$ is squarefree, we see easily that the first possibility in \eqref{Eq:possiblesquares} is impossible.

Suppose next that $4b+2\sqrt{-b(a^2-4b)}$ is a nonzero square, which implies that $-b(a^2-4b)$ is a square. Let \[d:=\gcd(-b,a^2-4b).\]

Consider first the case when $-b>0$.
 If $d=1$, then $a^2-4b$ is a square, since $-b(a^2-4b)$ is a square, which yields the contradiction that $\FF(x)=(x^4)^2+a(x^4)+b$ is reducible. Thus, $d>1$. Suppose that $d$ has exactly $n\ge 1$ distinct prime divisors
$r_1<r_2<\cdots <r_n$. Then $d=\prod_{i=1}^nr_i$ since $b$ is squarefree. %Note that $q_i\mid a$ for all $i$.
 Futhermore, since $-b(a^2-4b)$ is a square and $\gcd(-b/d,(a^2-4b)/d)=1$, it follows that %we can write
\begin{equation*}\label{Eq:S1S2}
-b=d \quad \mbox{and} \quad a^2-4b=Sd,
\end{equation*}
 where $S$ is a square.
 Thus,
 \begin{equation}\label{Eq:a^2}
   a^2=(S-4)d,
 \end{equation} which implies that $S\ge 4$.  Let $q$ be a prime divisor of $S$, so that $q^2\mid (a^2-4b)$ since $S$ is a square. Suppose that $q\ge 3$. Observe that if $q\mid a$ or $q\mid b$, then $q^2\mid b$, which contradicts the fact that $b$ is squarefree. Therefore, it must be that $q=2$, and we can assume that $S=2^{2k}$, for some integer $k\ge 1$. Thus, $2\mid a$, and  we have from \eqref{Eq:a^2} that
 \begin{equation}\label{Eq:a2}
 (a/2)^2=d(4^{k-1}-1).
 \end{equation} %Thus, $2\mid a$ and $2\mid (a^2-4b)$.

 Suppose that $2\mid b$. Then $r_1=2$ and
 \begin{equation}\label{Eq:a}
 (a/2)^2=2\left(\prod_{i=2}^nr_i\right)\left(4^{k-1}-1\right).
 \end{equation} If $k>1$, then $4^{k-1}-1\not \equiv 0\pmod{2}$ which implies that \eqref{Eq:a} is impossible, since the right-hand side of \eqref{Eq:a} is not a square. If $k=1$, then $a=0$, $-b(a^2-4b)=4d^2$ and
 \begin{equation}\label{Eq:contra}
 4b+2\sqrt{-b(a^2-4b)}=-4d+2(2d)=0,
 \end{equation}
 contradicting the fact that $4b+2\sqrt{-b(a^2-4b)}$ is a nonzero square.

 We conclude therefore that $2\nmid b$. It $k>1$, then we see from \eqref{Eq:a2} that $2\mid \mid a$ and $(a/2)^2\equiv b\equiv 1 \pmod{4}$.   Thus, condition \eqref{JKS:I3} of Theorem \ref{Thm:JKS} is false since \[\frac{a}{2}\equiv 1 \pmod{2} \quad \mbox{and} \quad -b\left(\frac{a}{2}\right)^2-\left(-\frac{b+b^8}{2}\right)^2\equiv 0 \pmod{2}.\]
   It follows that there are no monogenic polynomials $\FF(x)$ in the case when $-b>0$.

   Consider next the case that $b>0$. Then $4b-a^2>0$, and since $-b(a^2-4b)=b(4b-a^2)$ is a square, we get the contradiction that $b$ is a square if $d=1$. Hence, $d>1$. Arguing as in the case $-b>0$, %we can assume that $d=\prod_{i=1}^nr_i$,
   %where each $r_i$ is a prime with $r_1<r_2<\cdots  <r_n$, and
   we arrive at
    \begin{equation}\label{Eq:a^2b>0}
   (a/2)^2=(1-4^{k-1})d.
 \end{equation}
 Thus, $k=1$ and $a=0$. Then
      \[4b+2\sqrt{b(4b-a^2)}=8b,\] which implies that $b=2$ since $8b$ is a nonzero square and $b$ is squarefree. It is straightforward to verify that  $\FF(x)=x^8+2$ is monogenic and $\Gal(\FF)\simeq $8T6. Thus, there is only one mongenic polynomial $\FF(x)$ in this case.

   The arguments for the final possibility from \eqref{Eq:possiblesquares} are similar to the previous possibility, and so we omit the details. No further monogenic polynomials exist for this possibility.

  \subsection*{X=8} Following Theorem \ref{Thm:Awtrey}, we assume that none of $b$, $b(a^2-4b)$ and $4b-a^2$ is a square, that either $-b$ or $-b(a^2-4b)$ is a square, and that $\Gal(\FF)\simeq $ 8T8. Then either
  \begin{equation*}\label{Eq:PossSquares}
    2(a^2-4b)\sqrt{-b} \quad \mbox{or} \quad -4b+2\sqrt{-b(a^2-4b)}
  \end{equation*} is a nonzero square. %$-b$ or $-b(a^2-4b)$ is a square.

  We assume first that $2(a^2-4b)\sqrt{-b}$ is a nonzero square. Then $-b$ is a square. %, and since $\FF(x)=x^8+ax^4+b$ is irreducible, we have that $a^2-4b$ and, consequently, both $-b(a^2-4b)$ and  $-4b+2\sqrt{-b(a^2-4b)}$ are not squares. %Thus, $2(a^2-4b)\sqrt{-b}$ is a nonzero square.
  Suppose that $q$ is a prime dividing $b$. Then $q^2\mid b$. Thus,  condition \eqref{JKS:I1} of Theorem \ref{Thm:JKS} is false if $q\mid a$, while condition \eqref{JKS:I3} of Theorem \ref{Thm:JKS} is false if $q\nmid a$. Hence, $b=-1$, and $2(a^2+4)$ is a square. Suppose that $q\ge 3$ is a prime divisor of $a^2+4$. Then $q^2\mid (a^2+4)$ and $q\nmid 2a$, which implies that condition \eqref{JKS:I5} of Theorem \ref{Thm:JKS} is false. Therefore, we may assume that $a^2+4=2^k$, for some odd integer $k\ge 3$. Then, if $k>3$, we have that
  \[(a/2)^2=2^{k-2}-1\equiv 3 \pmod{4},\]
  which is impossible. Hence, $k=3$ and $a=\pm 2$. It is easy to verify that the two trinomials $\FF(x)=x^8\pm 2x^4-1$ are indeed monogenic and $\Gal(\FF)\simeq $ 8T8. However, Magma confirms that they generate isomorphic octic fields, distinct from the octic field generated by $\FF(x)=x^8-2$.

  Next, we assume that $-4b+2\sqrt{-b(a^2-4b)}$ is a nonzero square. It follows that $-b(a^2-4b)$ is a square. %Then, since $\FF(x)=x^8+ax^4+b$ is irreducible, we have that $a^2-4b$ and, consequently, both $-b$ and  $2(a^2-4b)\sqrt{-b}$ are not squares. Thus, $-4b+2\sqrt{-b(a^2-4b)}$ must be a nonzero square.
  Suppose that $q$ is a prime divisor of $b$. If $q^2\mid b$, then condition \eqref{JKS:I1} of Theorem \ref{Thm:JKS} is false if $q\mid a$, while condition \eqref{JKS:I3} of Theorem \ref{Thm:JKS} is false if $q\nmid a$. Hence, we assume that $b$ is squarefree. Thus, if $2\mid b$, then $2\mid (a^2-4b)$ since $-b(a^2-4b)$ is a square. Hence, $2\mid a$. Suppose that $2\nmid b$. If $2\nmid a$, then $-b(a^2-4b)\equiv 1 \pmod{4}$, so that
  \[-4b+2\sqrt{-b(a^2-4b)} \equiv 2 \pmod{4},\]
  which is impossible since $-4b+2\sqrt{-b(a^2-4b)}$ is a square. Hence, $2\mid a$ for any value of $b$. Also, if $b>0$, then
  $4b^2-a^2b<4b^2$, so that
  \[-4b+2\sqrt{-b(a^2-4b)}<-4b+2\sqrt{4b^2}=0,\]
  which is impossible since $-4b+2\sqrt{-b(a^2-4b)}$ is a square. Thus, $b<0$.

  Next, let $q\ge 3$ be a prime divisor of $a^2-4b$, and suppose that $q^2\mid (a^2-4b)$. Observe that if $q\mid a$ or $q\mid b$, then $q^2\mid b$, which contradicts the fact that $b$ is squarefree. Hence, since $2\mid a$, we have that
  \[a^2-4b=2^k \quad \mbox{or} \quad a^2-4b=2^k\prod_{i=1}^nr_i,\]
  where $k\ge 2$, $n\ge 1$ and the $r_i$ are distinct odd primes. Note that if $a^2-4b=2^k$ with $2\mid k$, then $-b=1$ since $b$ is squarefree and $-b(a^2-4b)$ is a square, which contradicts the fact that $-b$ is not a square. Therefore, we have that $2\nmid k$ for the possibility $a^2-4b=2^k$. Hence, since
  \[-b(a^2-4b)=-b\left(2^k\prod_{i=1}^nr_i\right),\]
  is a square, it follows that the possibilities for $-b$ are
  \[-b\in \left\{2,\quad 2\prod_{i=1}^nr_i,\quad \prod_{i=1}^nr_i\right\}.\]

  The possibility $-b=2$ occurs if $a^2-4b=2^k$, where $2\nmid k$. That is, we have $a^2=2^3(2^{k-3}-1)$, which implies that $k=3$, $a=0$ and $\FF(x)=x^8-2$. It is easily verified that $\FF(x)$ is monogenic with $\Gal(\FF)\simeq$ 8T8.

  The next possibility is $-b=2\prod_{i=1}^nr_i$ yields
  \[a^2=2^2\left(-2b\right)\left(2^{k-3}-1\right),\] which implies that $k=3$ and $a=0$. But then,
  \[-4b+2\sqrt{-b(a^2-4b)}=2^3(-2b),\]
  which contradicts the fact that $-4b+2\sqrt{-b(a^2-4b)}$ is a square. Hence, there are no monogenic polynomials for this possibility.

  The last possibility is $-b=\prod_{i=1}^nr_i$. In this case, we have that
  \[a^2=2^2(-b)\left(2^{k-2}-1\right)\]
   and
   \[-4b+2\sqrt{-b(a^2-4b)}=2^2(-b)\left(2^{(k-2)/2}+1\right).\] Hence,
   \[A:=(-b)\left(2^{k-2}-1\right) \quad \mbox{and}\quad B:=(-b)\left(2^{(k-2)/2}+1\right),\]
   are squares. Thus, $AB$ is a square, which implies that
   \[\left(2^{k-2}-1\right)\left(2^{(k-2)/2}+1\right)=\left(2^{(k-2)/2}-1\right)\left(2^{(k-2)/2}+1\right)^2,\]
   is a square, which in turn implies that $2^{(k-2)/2}-1$ is a square. By Theorem \ref{Thm:Ljunggren}, it follows that $k\in \{2,4\}$. When $k=2$, we get that $B=-2b$, which contradicts the fact that $B$ is a square. When $k=4$, we get that $b=-3$, $a=\pm 6$ and $\FF(x)=x^8\pm 6x^4-3$. Although $\Gal(\FF)\simeq$ 8T8 for both of these polynomials, neither one is monogenic, which can be seen in the following way. Observe that $2\mid a$ and $2\nmid b$. Hence, since
   \[a_2=a/2 \equiv 1 \Mod{2} \quad \mbox{and} \quad b_1=\frac{-b+(-b)^8}{2}=\frac{-3+(-3)^8}{2}\equiv 1 \Mod{2},\] we see that  $-ba_2^2-b_1^2\equiv 0 \Mod{2}$, and therefore,  condition \eqref{JKS:I2} of Theorem \ref{Thm:JKS} is false with $q=2$.

   In conclusion, there are exactly two monogenic polynomials
   \[\FF(x)\in \{x^8-2x^4-1, \quad x^8-2\}\] with $\Gal(\FF)\simeq$ 8T8.

\subsection*{X=15} Following Theorem \ref{Thm:Awtrey}, we assume that none of $b$, $b(a^2-4b)$ and $4b-a^2$ is a square, and that $\Gal(\FF)\simeq$ 8T15. Then, either $-b$ or $-b(a^2-4b)$ is a square, and none of
\[2\sqrt{-b},\quad 2(a^2-4b)\sqrt{-b},\quad 4b+2\sqrt{-b(a^2-4b)},\] \[4b-2\sqrt{-b(a^2-4b)}, \quad -4b+2\sqrt{-b(a^2-4b)}\]
is a nonzero square.

We assume first that $-b$ is a square. Suppose that $q$ is a prime dividing $b$. Then $q^2\mid b$. Thus,  condition \eqref{JKS:I1} of Theorem \ref{Thm:JKS} is false if $q\mid a$, while condition \eqref{JKS:I3} of Theorem \ref{Thm:JKS} is false if $q\nmid a$. Hence, $b=-1$. Let $D:=a^2+4$, and let $q\ge 3$ be a prime divisor of $D$. Note that $q\nmid 2a$. If $q^2\mid D$, then condition \eqref{JKS:I5} of Theorem \ref{Thm:JKS} is false. Thus, we may assume that $D/2^{\nu_2(D)}$ is squarefree. If $4\mid a$, then $a_2\equiv 0\pmod{2}$, and since $b_1=0$, we see that condition \eqref{JKS:I2} of Theorem \ref{Thm:JKS} is false with $q=2$. Hence, we can also assume that $4\nmid a$. Thus,
\begin{equation}\label{Eq:D8T15}
D=\left\{\begin{array}{cl}
  \prod_{i=1}^nr_i & \mbox{if $a\equiv 1 \Mod{2}$}\\[1em]
  2^3\prod_{i=1}^nr_i & \mbox{if $a\equiv 2 \Mod{4}$,}\\
\end{array}\right.
\end{equation}
where $n\ge 1$ and the $r_i$ are distinct odd primes. Define
\begin{equation}\label{Eq:G8T15}
G(t):=\left\{\begin{array}{cl}
  t^2+4 & \mbox{if $a\equiv 1 \Mod{2}$}\\[1em]
  t^2+1 & \mbox{if $a\equiv 2 \Mod{4}$.}\\
\end{array}\right.
\end{equation}
By Lemma \ref{Lem:ObstructionCheck}, we only have to check for obstructions at the prime $\ell=2$. Since $G(1)\not \equiv 0 \pmod{4}$ in each case of $G(t)$ in \eqref{Eq:G8T15}, we deduce that $G(t)$ has no obstructions. Hence, from Corollary \ref{Cor:Squarefree}, it follows  that there exist infinitely many values of $a$, in each case of $a$ in \eqref{Eq:D8T15}, for which such a value of $D$ exists.

We claim if $4\nmid a$, $b=-1$ and $D$ is as defined in \eqref{Eq:D8T15}, together with the restrictions set forth above from Theorem \ref{Thm:Awtrey}, then $\FF(x)=x^8+ax^4-1$ is monogenic with $\Gal(\FF)\simeq$ 8T15. From \eqref{Eq:DelF}, we have that $\Delta(\FF)=-2^{16}D^4$. To establish the claim, suppose first that $q$ is a prime divisor of $D$. If $q\ge 3$, then $q\nmid 2ab$, and it is easy to see that condition \eqref{JKS:I5} and, consequently, all other conditions of Theorem \ref{Thm:JKS} are true for $q$. If $q=2$, then $2\mid a$ and we see that $2\nmid a_2$ and $2\nmid (a_2^2-b_1^2)$ since $b_1=0$. Hence, in particular, condition \eqref{JKS:I2} of Theorem \ref{Thm:JKS} is true, and so all conditions of Theorem \ref{Thm:JKS} are true. Finally, suppose that $q=2$ and $2\nmid D$. Then $2\nmid a$ and we examine condition \eqref{JKS:I4} of Theorem \ref{Thm:JKS}, where
\[ H_1(x):=x^2+ax-1 \quad \mbox{and}\quad H_2(x):=\dfrac{ax^4-1+\left(-ax+1\right)^4}{2}.\] Since $H_1(x)\equiv x^2+x+1 \pmod{2}$ is irreducible in $\F_2[x]$ and
\[H_2(x)\equiv \left\{\begin{array}{cl}
  x^2(x+1)^2 \pmod{2} & \mbox{if $a\equiv 1 \Mod{4}$}\\[1em]
  x^2 \pmod{2} & \mbox{if $a\equiv 3 \Mod{4}$,}\\
\end{array}\right.\] it follows that $H_1(x)$ and $H_2(x)$ are coprime in $\F_2[x]$, which establishes the claim. Moreover, this argument proves that these polynomials account for all monogenics $\FF(x)$ with $\Gal(\FF)\simeq$ 8T15.

Finally, with $a>0$, together with the restrictions on $a$ set forth above, we claim that the set $\S:=\{\FF(x)=x^8+ax^4-1\}$ is an infinite family of monogenic 8T15-octic trinomials. The additional restriction of $a>0$ is due to the fact that the octic field generated by $x^8-ax^4-1$ is isomorphic to the octic field generated by $x^8+ax^4-1$. Then, comparing discriminants using Maple, it is easy to confirm that the trinomials in $\S$ generate distinct octic fields.

\subsection*{X=16} Following Theorem \ref{Thm:Awtrey}, we assume that $b$ is not a square and $\Gal(\FF)\simeq$ 8T16, so that $b(a^2-4b)$ is a square. Let $d=\gcd(b,a^2-4b)$, and let $q$ be a prime divisor of $d$. Then $q\mid a$. If $q^2\mid b$, then condition \eqref{JKS:I1} of Theorem \ref{Thm:JKS} is false. Hence, we can assume that $q\mid \mid b$ so that $d$ is squarefree. Next, suppose that $q\ge 3$ is a prime dividing $\widehat{b}:=b/d$, so that $q\nmid a$. If $q^2\mid b$, then we see that condition \eqref{JKS:I3} of Theorem \ref{Thm:JKS} is false since $a_1=0$ and $b_2=b/q\equiv 0\pmod{q}$. Thus, we may assume that $\widehat{b}/2^{\nu_2(\widehat{b})}$ is squarefree. Next, suppose that $q\ge 3$ is a prime divisor of $\widehat{d}:=(a^2-4b)/d$. Then $q\nmid a$ and so $q\nmid 2ab$. Hence, if $q^2\mid \widehat{d}$, then condition \eqref{JKS:I5} of Theorem \ref{Thm:JKS} is false. Therefore, we can assume that $\widehat{d}/2^{\nu_2(\widehat{d})}$ is squarefree. Thus, since $b(a^2-4b)$ is a square,  it follows that
%\[b=2^{c}d\prod_{i=1}^nr_i\quad \mbox{and} \quad a^2-4b=d\widehat{d}=2^{e}d\prod_{i=n+1}^mr_i,\] where the $r_i$ are distinct odd primes such that
%\[r_1<r_2<\cdots <r_n \quad \mbox{and} \quad r_{n+1}<r_{n+2}<\cdots <r_m.\]
\begin{equation}\label{Eq:b and a^2-4b}
b=d\widehat{b}=2^{c}d\quad \mbox{and} \quad a^2-4b=d\widehat{d}=2^{e}d
\end{equation} for some nonegative integers $c$ and $e$ with $c\equiv e\pmod{2}$.

If $2\mid d$, then $2\mid a$ so that $2^2\mid (a^2-4b)$. Note then that $2^2\nmid b$, since $d$ is squarefree. Thus, $2\mid \mid b$ and $b=d$. Hence, since $b(a^2-4b)$ is a square, we have that $2^3\mid (a^2-4b)$, which implies that $2^4\mid a^2$. However, $2^4\nmid (a^2-4b)$ since $2\mid \mid b$. Thus, $2^3\mid \mid (a^2-4b)$.  Therefore, in this case, we have that $c=0$, $e=2$ and $a^2=8d$, which implies that $b=d=2$ and $a=\pm 4$. It is straightforward to verify that the two polynomials $\FF(x)=x^8\pm 4x^4+2$ are both monogenic with $\Gal(\FF)\simeq$ 8T16.

Suppose next that $2\nmid d$. There are three cases to consider:
\begin{enumerate}
  \item \label{Case1} $2\mid a$ and $2\nmid b$,
  \item \label{Case2} $2\nmid a$ and $2\mid b$,
   \item \label{Case3} $2\nmid a$ and $2\nmid b$.
    \end{enumerate}

For case \eqref{Case1}, we have from \eqref{Eq:b and a^2-4b} that $b=d$ and
\begin{equation}\label{Eq:a^28T16}
a^2=2^2(2^{e-2}+1)d,
\end{equation} where $e\ge 2$. Observe that the right-hand side of \eqref{Eq:a^28T16} is not a square if $e=2$. If $e=3$ in \eqref{Eq:a^28T16}, then $d=3=b$,  $a=\pm 6$, and $b(a^2-4b)=72$, contradicting the fact that $b(a^2-4b)$ is a square. If $e\ge 4$ in \eqref{Eq:a^28T16}, then $4\nmid a$ and $b\equiv 1 \pmod{4}$. An examination of condition \eqref{JKS:I2} of Theorem \ref{Thm:JKS} reveals then that $2\nmid a_2$ and $2\nmid b_1$, so that $2\mid (a_2((-b)a_2^2-b_1^2)$, which implies that condition \eqref{JKS:I2} is false. Thus, there are no monogenic polynomials in this case.

For case \eqref{Case2}, we see that $2\nmid (a^2-4b)$. Hence, from \eqref{Eq:b and a^2-4b}, we have that $b=2^cd$, for some $c\ge 1$, and $a^2-4b=d$. Then, since $2\nmid d$ and $b(a^2-4b)=2^cd^2$ is a square, we deduce that $c\equiv  0\pmod{2}$ with $c\ge 2$. Thus, $2\mid b/2$, which implies that condition \eqref{JKS:I3} of Theorem \ref{Thm:JKS} is false, and there are no monogenic polynomials in this case.

Finally, for case \eqref{Case3}, we see from \eqref{Eq:b and a^2-4b} that $b=d=a^2-4b$. Hence, $a^2=5d$, which implies that $d=b=5$, since $d$ is squarefree,  and $a=\pm 5$. Since $2\nmid ab$, we examine condition \eqref{JKS:I4} of Theorem \ref{Thm:JKS} with $q=2$ for each of the two polynomials $\FF(x)=x^8\pm 5x^4+5$. For $\FF(x)=x^8+5x^4+5$, we get
\[H_1(x)\equiv x^2+x+1 \pmod{2} \quad \mbox{and} \quad H_2(x)\equiv (x^2+x+1)^2 \pmod{2},\] so that $H_1(x)$ and $H_2(x)$ are not coprime in $\F_2[x]$. Hence, $\FF(x)=x^8+5x^4+5$ is not monogenic. For $\FF(x)=x^8-5x^4+5$, we get
\[H_1(x)\equiv x^2+x+1 \pmod{2} \quad \mbox{and} \quad H_2(x)\equiv (x+1)^2 \pmod{2},\] so that $H_1(x)$ and $H_2(x)$ are coprime in $\F_2[x]$. Hence, $\FF(x)=x^8-5x^4+5$ is monogenic with $\Gal(\FF)\simeq$ 8T16.

Note that
\[\Delta(x^8-4x^4+2)=2^{31}=\Delta(x^8+4x^4+2),\]
and $\Delta(x^8-5x^4+5)=2^{16}5^7$, so that the octic field generated by $x^8-5x^4+5$ is distinct from the other two. All zeros of $x^8-4x^4+2$ are real, while all zeros of $x^8+4x^4+2$ are non-real, and Magma confirms that the respective octic fields are not isomorphic. Therefore, in summary, there are exactly three monogenic polynomials
\[\FF(x)\in \{x^8-4x^4+2,\quad x^8+4x^4+2,\quad x^8-5x^4+5\}\]
with $\Gal(\FF)\simeq$ 8T16.

\subsection*{X=17} Following Theorem \ref{Thm:Awtrey}, we assume that $b$ is not a square and $\Gal(\FF)\simeq$ 8T17, so that $b(a^2-4b)$ is not a square and $4b-a^2$ is a square. Note then that $a\ne 0$ since $b$ is not a square. Let $d=\gcd(b,4b-a^2)$. Suppose that $d>1$ and let $q$ be a prime divisor of $d$. Then $q\mid a$. If $q^2\mid b$, then condition \eqref{JKS:I1} of Theorem \ref{Thm:JKS} is false. Hence, we can assume that $q\mid \mid b$, so that $d$ is squarefree. Furthermore, if $q\ge 3$, then since $q^2\mid a^2$ and $q\mid \mid b$, it follows that $q^2\nmid (4b-a^2)$, contradicting the fact that $4b-a^2$ is a square. Thus, $d\in \{1,2\}$.

Since
 \[4b-a^2\equiv -a^2\pmod{4},\] and $4b-a^2$ is a square, it follows that $2\mid a$ and $b-(a/2)^2$ is a square.
 Suppose that $2\mid (b-(a/2)^2)$. If $2\mid (a/2)$, then $2^2\mid (a/2)^2$. But $2^2\nmid b$ so that $2^2\nmid (b-(a/2)^2)$, contradicting the fact that $b-(a/2)^2$ is a square. Thus, $2\nmid (a/2)$ and $2\nmid b$. Then, $(a/2)^2\equiv 1 \pmod{4}$, and since $2\mid (b-(a/2)^2)$ and $b-(a/2)^2$ is a square, we deduce that $2^2\mid (b-(a/2)^2)$ so that $b\equiv 1 \pmod{4}$. Then,
 \[b_1=\frac{b+(-b)^8}{2}\equiv 1 \pmod{4}\] and
 \[(-b)(a/2)^2-b_1^2\equiv 0\pmod{2},\] which implies that condition \eqref{JKS:I2} of Theorem \ref{Thm:JKS} is false. Hence, we can assume that $2\nmid (b-(a/2)^2)$, and therefore, $2^2\mid \mid (4b-a^2)$.

 Next, suppose that $q\ge 3$ is a prime divisor of $4b-a^2$. We have shown above that $q\nmid b$, and therefore $q\nmid a$. Thus $q\nmid 2ab$. However, since $4b-a^2$ is a square, we have that $q^2\mid (4b-a^2)$. Thus, condition \eqref{JKS:I5} of Theorem \ref{Thm:JKS} is false. Hence, for the monogenicity of $\FF(x)$, we must have that $4b-a^2=4$.

 Consequently, we have shown that if $\FF(x)=x^8+ax^4+b$ is monogenic with $\Gal(\FF)\simeq$ 8T17, then $a=2t$ and $b=t^2+1$, for some integer $t\ne 0$. In fact, with $t\ne 0$, we claim that
 \begin{equation}\label{Eq:Claim8T17}
 \mbox{$\FF(x)=x^8+2tx^4+t^2+1$ is monogenic if and only if $t^2+1$ is squarefree.}
  \end{equation} To establish \eqref{Eq:Claim8T17}, we begin by noting that $\FF(x)=x^8+2tx^4+t^2+1$ is irreducible by \cite[Lemma 3.4]{HJOkayama} if $t^2+1$ is squarefree, since
  \[(2t \mmod{4},\ t^2+1 \mmod{4})\in \{(0,1),(2,2)\}\] and $t^2+1\ge 2$. Note, by \eqref{Eq:DelF}, that
  \begin{equation}\label{Eq:Dis8T17}
  \Delta(\FF)=2^{24}(t^2+1)^3.
   \end{equation} Observe that if $q=2$ is a divisor of $t^2+1$, then, since $q\mid 2t$, condition \eqref{JKS:I1} of Theorem \ref{Thm:JKS} is false if and only if $4\mid (t^2+1)$. If $q=2$ does not divide $t^2+1$, then $t^2\equiv 0 \pmod{4}$, and in this case, we examine condition \eqref{JKS:I2} of Theorem \ref{Thm:JKS}. We see that
  \[a_2=t\equiv 0 \pmod{2}\quad \mbox{and} \quad b_1=\frac{t^2+1+(-(t^2+1))^8}{2}\equiv 1 \pmod{2}.\] Thus, condition \eqref{JKS:I2} of Theorem \ref{Thm:JKS} is true. Suppose next that $q\ge 3$ is a prime divisor of $t^2+1$. Then $q\nmid 2t$. Thus, we examine condition \eqref{JKS:I3} of Theorem \ref{Thm:JKS}. Since $q$ is odd, we see that $a_1=0$ and therefore, the second statement under condition \eqref{JKS:I3} is false. So, to determine whether condition \eqref{JKS:I3} of Theorem \ref{Thm:JKS} is true or false, we need to examine the first statement under condition \eqref{JKS:I3}. Thus, since $b_2=(t^2+1)/q$, we see that condition \eqref{JKS:I3} of Theorem \ref{Thm:JKS} is false if and only if $q^2\mid (t^2+1)$, which establishes the claim  \eqref{Eq:Claim8T17}.

  Since, $G(1)\not \equiv 0 \pmod{4}$ for $G(t)=t^2+1$, we conclude from Lemma \ref{Lem:ObstructionCheck} that $G(t)$ has no local obstructions, and therefore, from Corollary \ref{Cor:Squarefree}, we have that there exist infinitely many values of $t$ such that $G(t)$ is squarefree. Thus, we have shown that there exist infinitely many monogenic trinomials $\FF(x)=x^8+2tx^4+t^2+1$ with $\Gal(\FF)\simeq$ 8T17. Moreover, comparing discriminants \eqref{Eq:Dis8T17} for all such $t>0$, we see that the octic fields they generate are distinct.

  \subsection*{X=26}Computer evidence suggests that there is an abundance of monogenic polynomials $\FF(x)$ with $\Gal(\FF)\simeq$ 8T26. We use a strategy similar to the case of 8T9 to construct one infinite family. %In particular, we show that there exist infinitely many prime values of the parameter $t$ such that
  %\begin{equation*}\label{Eq:Claim8T26}
  %\mbox{$\FF(x)=x^8+(4t+3)x^4+(4t+3)$ is monogenic with $\Gal(\FF)\simeq$ 8T26.}
  %\end{equation*}
  Let
  %\begin{equation*}\label{Eq:Del8T26}
  $\FF(x)=x^8+tx^4+3$.
    %\end{equation*}
  By \eqref{Eq:DelF}, we have that
  \begin{equation}\label{Eq:Del8T26}
  \Delta(\FF)=2^{16}3^3(t^2-12)^4.
  \end{equation}
  Let $G(t)=t^2-12$. Since $G(1)\equiv 1 \pmod{4}$, we see from Lemma \ref{Lem:ObstructionCheck} that $G(t)$ has no local obstructions. Hence, by Corollary \ref{Cor:Squarefree}, there exist infinitely many primes $p$ such that $G(p)$ is squarefree. Let $p\ge 5$ be such a prime. We claim that
  \begin{equation}\label{Eq:Claim8T26}
  \mbox{$\FF(x)=x^8+px^4+3$ is monogenic with $\Gal(\FF)\simeq$ 8T26.}
  \end{equation}
Since
\[(p \Mod{4},\ 3 \Mod{4})\in \{(1,3),(3,3)\},\] we have that $\FF(x)$ is irreducible by \cite[Lemma 3.4]{HJOkayama}. With $a=p$ and $b=3$, we have that
\begin{align*}
b&=3,\\
b(a^2-4b)&=3(p^2-12)\not\equiv 0 \Mod{9},\\
4b-a^2&=-(p^2-12)<0,\\
-b&=-3<0 \quad \mbox{and}\\
-b(a^2-4b)&=-3(p^2-12)<0.
\end{align*} Hence, it follows from Theorem \ref{Thm:Awtrey} that $\Gal(\FF)\simeq$ 8T26.

To see that $\FF(x)$ is monogenic, first let $q=2$. Since $2\nmid p$, we examine condition \eqref{JKS:I4} of Theorem \ref{Thm:JKS}. Then
\begin{align*}
H_1(x)&\equiv x^2+x+1 \pmod{2} \quad \mbox{and}\\
H_2(x)&=p(p^2-p+1)\left(\dfrac{p+1}{2}\right)x^4+6p^3x^3+27p^2x^2+54px+42\\
&\equiv x^2\left(\overline{\left(\dfrac{p+1}{2}\right)}x+1\right)^2 \pmod{2}.
\end{align*} Hence, $H_1(x)$ and $H_2(x)$ are coprime in $\F_2[x]$.

Next, let $q=3$. We examine condition \eqref{JKS:I3} of Theorem \ref{Thm:JKS}. Since $a_1=0$ and $b_2=1$, we see that the first statement under condition \eqref{JKS:I3} is true.

Finally, let $q$ be a prime divisor of $p^2-12$. Then, $q\nmid 6p$, and we see that condition \eqref{JKS:I5} of Theorem \ref{Thm:JKS} is true since $p^2-12$ is squarefree.  Therefore, the claim \eqref{Eq:Claim8T26} is established.

Suppose now that $p_1$ and $p_2$ are primes with $5\le p_1<p_2$ such that $G(p_1)$ and $G(p_2)$ are squarefree, and that $\FF_{p_1}(x)$ and $\FF_{p_2}(x)$ generate isomorphic octic fields. Then, since  $\FF_{p_1}(x)$ and $\FF_{p_2}(x)$ are both monogenic, we must have that
  \begin{equation}\label{Eq:Soldis8T26}
  2^{16}3^3(p_1^2-12)^4=\Delta(\FF_{p_1})=\Delta(\FF_{p_2})=2^{16}3^3(p_2^2-12)^4,
  \end{equation} which contradicts the fact that $5\le p_1<p_2$. %Aside from the solution of $p=p'$ to \eqref{Eq:Soldis8T26}, Maple gives several other solutions, all of which are impossible here. One solution has $p'=-p$, while two other solutions require that $24-p^2$ be a square, which is impossible since $p\ge 5$. The other solutions require that $(p^2-12)i+12$ be a rational square, which is also impossible.
  We deduce that the set
  \[\{\FF(x)=x^8+px^4+3: p\ge 5 \mbox{ is prime with $p^2-12$ squarefree}\},\] is an infinite family of monogenic 8T26-octic trinomials $\FF(x)$, which completes the proof of item \eqref{I:1}.

  We turn now to the proof of item \eqref{I:2}. We recall the definition of $W_1$, $W_2$ and $W_3$ given in \eqref{Eq:Widef}.
Straightforward computations reveal that
\[\Delta(g)=(b+2-2a)(b+2+2a)(a^2-4b+8)^2=W_1W_2W_3^2 \quad \mbox{and}\quad \Delta(\GG)=2^8\Delta(g)^2.\]

 \subsection*{X=2} We assume that $\Gal(\GG)\simeq$ 8T2. Following Theorem \ref{Thm:AP1}, we see that there are three possibilities.

 Suppose first that $W_1W_3$ and $W_2W_3$ are squares, while
 $W_1W_2W_3$ is not a square. Then $W_1W_3W_2W_3=W_1W_2W_3^2$ is a square, which implies that $W_1W_2$ is a square. %It follows that $W_1$, $W_2$ and $W_3$ are not squares.
 Note that $W_1W_2\ne 0$ since $\GG(x)$ is irreducible. Furthermore, $W_1\ne W_2$, since if $W_1=W_2$, then $a=0$, which contradicts our assumption that $a\ne 0$
   from \eqref{Eq:G}. Hence, since $W_1W_2$ is a square, we conclude that $W_1W_2>1$ and that there exists a prime divisor $q$ of $W_1W_2$, such that
    $q^2\mid W_1$ or $q^2\mid W_2$. Hence, by Lemma \ref{Lem:MainG2}, it follows that $\GG(x)$ is not monogenic.

The second possibility in Theorem \ref{Thm:AP1} is that $W_1W_3$ and $W_1W_2W_3$ are squares, while
 $W_2W_3$ is not a square. Then $W_1W_3W_1W_2W_3=W_1^2W_2W_3^2$ is a square, which implies that $W_2$ is a square. % but $W_1$, $W_1W_2$ and $W_3$ are not squares.
  If $W_2>1$, then $q^2\mid W_2$ for some prime $q$, and $\GG(x)$ is not monogenic by Lemma \ref{Lem:MainG2}. Therefore, suppose then that $W_2=1$. Hence, $b=-2a-1$ and
  \[W_1W_3=-4a^3-31a^2-40a+12.\]
  Since $W_1W_3$ is a square, say $m^2$, we consider the elliptic curve
  \begin{equation}\label{E1:X=2}
  y^2=x^3-32x^2+160x+192,
    \end{equation} where $x:=-4a$ and $y:=4m$.
     %\[x:=-4a \quad \mbox{and} \quad y:=4m.\]
     Using Sage to find all integral points $(x,y)$, with $y\ge 0$, on \eqref{E1:X=2} yields
     \[(x,y)\in \{(-1,0), (4,20), (8,0), (24,0), (44,180)\}.\] Checking the points $(x,y)$ where $x\equiv 0\pmod{4}$ reveals that all such points, with the exception of  $(x,y)=(4,20)$, produce values $(a,b)$ such that $\GG(x)$ is reducible. The lone point $(4,20)$ implies that $(a,b)=(-1,1)$, and it is straightforward to verify that $\GG(x)=x^8-x^6+x^4-x^2+1$ is indeed monogenic.

   The third and final possibility in Theorem \ref{Thm:AP1} is that $W_2W_3$ and $W_1W_2W_3$ are squares, while
 $W_1W_3$ is not a square. Arguing as before, we see that $W_1$ is a square, and if $W_1>1$, we achieve no monogenic polynomials by Lemma \ref{Lem:MainG2}. If $W_1=1$, then we arrive at the same elliptic curve \eqref{E1:X=2}, where $x:=4a$ and $y:=4m$.
 %where   \[x:=4a \quad \mbox{and} \quad y:=4m.\]
 In this case, the integral points with $x\equiv 0 \pmod{4}$ produce reducible polynomials $\GG(x)$, except when $(x,y)=(44,180)$. For this particular point, we get that $(a,b)=(11,21)$. However, is is easy to verify that, while the polynomial
  \[\GG(x)=x^8+11x^6+21x^4+11x^2+1\] is irreducible with $\Gal(\GG)\simeq$ 8T2, it
  is not monogenic.

 In conclusion, there is exactly one monogenic polynomial
 \[\GG(x)=\Phi_{10}(x^2)=x^8-x^6+x^4-x^2+1,\] where $\Gal(\GG)\simeq$ 8T2. %In summary,  $\GG(x)=x^8-x^6+x^4-x^2+1$ is the only monogenic polynomial when X=2.

 \subsection*{X=3} We assume that $\Gal(\GG)\simeq$ 8T3. Then, $W_1$, $W_2$ and $W_1W_2$ are all squares from Theorem \ref{Thm:AP1}. If $W_1=W_2$, then $a=0$, which contradicts the fact that $a\ne 0$ from \eqref{Eq:G}. Hence, without loss of generality, suppose that $W_1>1$. Then there exists a prime $q$ with $q^2\mid W_1$, and we can argue as in the case X=2 to deduce that no monogenic polynomials $\GG(x)$ exist in this case.

  \subsection*{X=4} We assume that $\Gal(\GG)\simeq$ 8T4. From Theorem \ref{Thm:AP2}, we have that exactly one of
\[W_1,\quad W_2 \quad \mbox{and} \quad W_1W_2 \quad \mbox{is a square,}\] and none of
\[W_1W_3,\quad W_2W_3 \quad \mbox{and} \quad W_1W_2W_3 \quad \mbox{is a square.}\] Suppose first that $W_1$ is a square. We deduce from Lemma \ref{Lem:MainG2} that $W_1=1$, so that $b=2a-1$. Then, it is easy to see that
  \[W_2\left(-a+4-2\sqrt{W_1}\right)=(4a+1)(-a+2)=-4a^2+7a+2\] is a square if and only if $a=2$. But then $W_3=a^2-4b+8=a^2-4=0$, which contradicts the fact that $W_1W_3=W_3$ is not a square. Also, it is easy to see that
    \[W_2\left(-a+4+2\sqrt{W_1}\right)=(4a+1)(-a+6)\] is a square if and only if $a\in \{1,2,6\}$. With the corresponding values of $b$, we observe that $W_3=0$ when $a\in \{2,6\}$, which again contradicts the fact that $W_1W_3$ is not a square. For $a=1$, we have that $\GG(x)=x^8+x^6+x^4+x^2+1$, which is reducible. Hence, there are no monogenics when $W_1$ is a square.

    The case when $W_2$ is a square is similar. We can assume that $W_2=1$ by Lemma \ref{Lem:MainG2}, so that $b=-2a-1$. Then, using Maple (the {\bf isolve} command), we see that
    \[W_1\left(-a-4-2\sqrt{W_2}\right)=(4a-1)(a+6)\] is a square if and only if $a\in \{-42,-11,-6\}$, while
    \[W_1\left(-a-4+2\sqrt{W_2}\right)=(4a-1)(a+2)\] is a square if and only if $a\in \{-6,-2,1\}$. Checking these values yields the contradictions that $W_3$ is a square when $a\in \{-6,-2\}$,  while $\GG(x)$ is reducible when $a\in \{-42,-11\}$. Although, $\GG(x)$ is irreducible when $a=1$, it is not monogenic. Thus, there are no monogenics arising in this situation.

    Finally, suppose that $W_1W_2$ is a square, and neither $W_1$ nor $W_2$ is a square. By Lemma \ref{Lem:MainG2}, we deduce that $W_1$ and $W_2$ are squarefree. Since $W_1W_2$ is a square, it follows that $W_1=W_2$, contradicting our assumption that $a\ne 0$ in \eqref{Eq:G}. Hence, there are no monogenic polynomials $\GG(x)$ when X=4.

\subsection*{X=9} Let
\[\GG_t(x):=x^8+(4t+3)x^6+(8t+5)x^4+(4t+3)x^2+1.\] Then
$\Delta(\GG)=256(16t+13)^2(4t+1)^4(4t-3)^4$. Let $G(t):=(16t+13)(4t+1)(4t-3)$. Since $G(1)\equiv 1 \pmod{4}$, we deduce from Lemma \ref{Lem:ObstructionCheck} that $G(t)$ has no local obstructions. Hence, from Corollary \ref{Cor:Squarefree}, there exist infinitely many primes $p$ such that $G(p)$ is squarefree. Let $p$ be such a prime. Note, for $\GG_p(x)$, with $a:=4p+3$, $b:=8p+5$ and $W_i$ as defined in \eqref{Eq:Widef}, we have
\[W_1=1,\quad W_2=16p+13 \quad \mbox{and} \quad W_3=(4p+1)(4p-3),\]
so that
\[W_1W_2W_3=(16p+13)(4p+1)(4p-3)=G(p).\] Since $W_1W_2W_3$ is squarefree and
 \[(a \mmod{4}, \ b \mmod{4})=(3,1),\] it follows from Proposition \ref{Lem:BAMS} that $\GG_p(x)$ is monogenic. Moreover, since $W_1=1$ is a square and $G(p)$ is squarefree, it is easy to see that
 none of
 \[W_2, \ W_1W_2, \ W_1W_3, \ W_2W_3,\  W_1W_2W_3,\]
 \begin{align*}
   W_2\left(-a+4-2\sqrt{W_1}\right)&=-(16p+3)(4p+1)\quad \mbox{and}\\
   W_2\left(-a+4+2\sqrt{W_1}\right)&=-(16p+3)(4p-3)
 \end{align*} is a square. Hence, we deduce from Theorem \ref{Thm:AP2} that $\Gal(\GG_p)\simeq$ 8T9.

  Suppose that, for some primes $p<q$, with $G(p)$ and $G(q)$ both squarefree, we have that $\GG_p(x)$ and $\GG_q(x)$ generate the same octic field. Then, since $\GG_p(x)$ and $\GG_q(x)$ are both monogenic, it must be that $\Delta(\GG_p)=\Delta(\GG_q)$, which implies that
\[256(16p+13)^2(4p+1)^4(4p-3)^8=256(16q+13)^2(4q+1)^4(4q-3)^4,\] contradicting the fact that $p<q$.
Thus, $\{\GG_p(x): G(p) \mbox{ is squarefree}\}$
is an infinite family of monogenic even octic 8T9-polynomials.

 \subsection*{X=10} We assume that $\Gal(\GG)\simeq$ 8T10. Then, following Theorem \ref{Thm:AP1}, we have that exactly one of
 \[W_1W_3, \quad W_2W_3 \quad \mbox{and} \quad W_1W_2W_3 \quad \mbox{is a square.}\] Note then that $W_1\ne 1$, $W_2\ne 1$ and  $W_1\ne W_2$.  Since we are searching for monogenic polynomials, we can also assume that $W_1$ and $W_2$ are squarefree, by Lemma \ref{Lem:MainG2}.
 %where $W_1$, $W_2$ and $W_3$ are as defined in \eqref{Eq:Widef}.

 Suppose first that
 \begin{align}\label{Eq:8T10 con1}
 \begin{split}
 W_1W_3& \quad \mbox{is a square}\\
 \mbox{while neither}\quad W_2W_3& \quad \mbox{nor} \quad W_1W_2W_3 \quad \mbox{is a square.}
 \end{split}
 \end{align}
 Note that $W_1\mid W_3$ since $W_1$ is squarefree and $W_1W_3$ is a square.
  If $W_1=W_3=-1$, then $a=4\pm \sqrt{-5}$, while if $W_1=W_3=1$, then $a=4\pm \sqrt{5}$. Hence, $W_1W_3>1$. Suppose that $q$ is a prime divisor of $W_3$ such that $q\nmid W_1$. Since $W_1W_3$ is a square, and $W_1$ is squarefree, it follows that $q^2\mid W_3$. Thus, by Lemma \ref{Lem:MainG2}, we can assume that $q=2$. Then, since $2\mid W_3$ and $2\nmid W_1$, it follows that $2\mid a$ and $2\nmid b$. Thus,
     \[(a \mmod{4}, \ b \mmod{4})\in \{(0,1),(0,3),(2,1),(2,3)\}.\] A closer examination reveals that $2^3\mid \mid W_3$ when $(a \mmod{4}, \ b \mmod{4})=(2,1)$, which contradicts the fact that $W_1W_3$ is a square. Invoking Lemma \ref{Lem:MainG2} once again narrows it down to $(a \mmod{4}, \ b \mmod{4})=(0,3)$. However, in this case, it is easy to see that $W_1\equiv 1 \pmod{4}$ and $W_3/4\equiv 3 \pmod{4}$ so that $W_1W_3/4\equiv 3 \pmod{4}$, which contradicts the fact that $W_1W_3/4$ is a square. Consequently, we have shown that the prime divisors of $W_3$ are exactly the prime divisors of $W_1$, and furthermore,
     \[\mbox{either } W_3=W_1\quad \mbox{or} \quad W_3=2^{2k}W_1\quad \mbox{for some integer $k\ge 1$}.\]

     If $W_3=W_1$, then $b=((a+1)^2+5)/5$ so that
     %\[W_1=\dfrac{(a-4)^2}{5}.\]
     $W_1=(a-4)^2/5$. Since $W_1$ is squarefree, it follows that $a \in \{-1,9\}$. If $a=-1$, then $W_2=1$, which contradicts the fact that $W_2$ is squarefree. Thus, $a=9$, $b=21$ and $\GG(x)=x^8+9x^6+21x^4+9x^2+1$. We use Theorem \ref{Thm:Dedekind} with $T(x):=\GG(x)$ and $q=2$ to investigate the monogenicity of $\GG(x)$. Since $\overline{T}(x)=(x^4+x^3+x^2+x+1)^2$, we can let $h_1(x)=h_2(x)=x^4+x^3+x^2+x+1$. Then
     \begin{align*}
       F(x)&=\dfrac{(x^4+x^3+x^2+x+1)^2-(x^8+9x^6+21x^4+9x^2+1)}{2}\\
       &=x^7-3x^6+2x^5-8x^4+2x^3-3x^2+x\\
       &\equiv x^7+x^6+x^2+x \pmod{2}\\
       &\equiv x(x+1)^2(x^4+x^3+x^2+x+1).
     \end{align*}
     Hence, $\gcd(\overline{F},\overline{h_1})=x^4+x^3+x^2+x+1$ and $\GG(x)$ is not monogenic by Theorem \ref{Thm:Dedekind}. Thus, there are no monogenic 8T10-polynomials $\GG(x)$ when conditions \eqref{Eq:8T10 con1} hold with $W_1=W_3$.

     Suppose then that $W_3=2^{2k}W_1$ for some integer $k\ge 1$. Then
     \[b=\dfrac{a^2+2^{2k+1}a-2^{2k+1}+8}{2^{2k}+4},\]
     and since %so that
     %\begin{align*}
     %  W_1=\dfrac{(a-4)^2}{2^{2k}+4},& \quad W_2=\dfrac{a^2+\left(2^{2k+2}+8\right)a+16}{2^{2k}+4} \quad \mbox{and} \quad  %W_3=\dfrac{2^{2k}(a-4)^2}{2^{2k}+4}.
       %W_3=\dfrac{2^{2k}(a-4)^2}{2^{2k}+4}& \quad \mbox{and} \quad W_1W_3=\dfrac{2^{2k}(a-4)^4}{\left(2^{2k}+4\right)^2}
     %\end{align*} Since
     $b$ is an integer, we see that $2\mid a$. Suppose that $k\ge 2$. Since $2^2\mid\mid (2^{2k}+4)$, it follows that
     \begin{equation}\label{Eq:bcon}
     b\equiv \left\{\begin{array}{cl}
       1 \pmod{2} & \mbox{if $a\equiv 2 \pmod{4}$}\\[.5em]
       0 \pmod{2} & \mbox{if $a\equiv 0 \pmod{4}$.}
     \end{array}\right.
     \end{equation} We use Theorem \ref{Thm:Dedekind} with $T(x):=\GG(x)$ and $q=2$ to examine the monogenicity of $\GG(x)$. From \eqref{Eq:bcon}, we get that
     \[\GG(x)\equiv \left\{\begin{array}{cl}
       (x^2+x+1)^4 \pmod{2} & \mbox{if $a\equiv 2 \pmod{4}$}\\[.5em]
       (x+1)^8 \pmod{2} & \mbox{if $a\equiv 0 \pmod{4}$.}
     \end{array}\right.\]
    % \[a^2+2^{2k+1}a-2^{2k+1}+8\equiv a^2 \pmod{8} \quad \mbox{and} \quad 2^{2k}+4\equiv 4 \pmod{8}.\]
    % \[a^2+2^{2k+1}a-2^{2k+1}+8\equiv a^2 \pmod{8} \quad \mbox{and} \quad 2^{2k}+4\equiv \left\{\begin{array}{cl}
    %    0 \pmod{8} & \mbox{when $k=1$}\\[.25em]
    %    4 \pmod{8} & \mbox{when $k\ge 2$.}
    %  \end{array}\right.\]

    Hence, if $a\equiv 2 \pmod{4}$, then $a^2+2^{2k+1}a-21\cdot 2^{2k}-68\equiv 0 \pmod{16}$, and therefore
    \begin{multline*}
       F(x)=2x^7+\left(\frac{10-a}{2}\right)x^6+8x^5-\left(\frac{a^2+2^{2k+1}a-21\cdot 2^{2k}-68}{2(2^{2k}+4)}\right)x^4\\
    +8x^3+\left(\frac{10-a}{2}\right)x^2+2x\equiv 0 \pmod{2}.
    \end{multline*}
    %\begin{align*}
    %F(x)&=2x^7+\left(\frac{10-a}{2}\right)x^6+8x^5-\left(\frac{a^2+2^{2k+1}a-21\cdot 2^{2k}-68}{2(2^{2k}+4)}\right)x^4\\
    %&\qquad +8x^3+\left(\frac{10-a}{2}\right)x^2+2x\\
    %&\equiv 0 \pmod{2}.
    %\end{align*}
    Thus, $\gcd(\overline{F},x^2+x+1)\ne 1$ and $\GG(x)$ is not monogenic.

    If $a\equiv 0 \pmod{4}$, then $a^2+2^{2k+1}a-72\cdot 2^{2k}-272\equiv 0 \pmod{16}$, and hence
    \begin{multline*}
      F(x)= \left(\frac{4^{k+1}+16}{2^{2k}+4}\right)x^7+\left(\frac{28-a}{2}\right)x^6+28x^5-\left(\frac{a^2+2^{2k+1}a-72\cdot 2^{2k}-272}{2(2^{2k}+4)}\right)x^4\\
    +28x^3+\left(\frac{28-a}{2}\right)x^2+ \left(\frac{4^{k+1}+16}{2^{2k}+4}\right)x\equiv 0 \pmod{2}.
    \end{multline*}
    %\begin{align*}
    %F(x)=& \left(\frac{4^{k+1}+16}{2^{2k}+4}\right)x^7+\left(\frac{28-a}{2}\right)x^6+28x^5-\left(\frac{a^2+2^{2k+1}a-72\cdot %2^{2k}-272}{2(2^{2k}+4)}\right)x^4\\
    %&\qquad +28x^3+\left(\frac{28-a}{2}\right)x^2+ \left(\frac{4^{k+1}+16}{2^{2k}+4}\right)x\\
     %&\equiv 0 \pmod{2}.
    %\end{align*}
    Thus, $\gcd(\overline{F},x+1)\ne 1$ and $\GG(x)$ is not monogenic.

     Therefore, $\GG(x)$ is not monogenic when $k\ge 2$. So, suppose now that $k=1$. Then $b=a^2/8+a$ so that $4\mid a$. Then
     \[\GG(x)=x^8+ax^6+(a^2/8+a)x^4+ax^2+1\equiv (x+1)^8\pmod{2}.\] Hence, applying Theorem \ref{Thm:Dedekind} with $T(x):=\GG(x)$ and $q=2$, we get that
     \begin{multline*}
     F(x)=4x^7+\left(14-a/2\right)x^6+28x^5-\left(a^2/16+a/2-35\right)x^4\\
     +28x^3+\left(14-a/2\right)x^2+4x \equiv 0 \pmod{2}
     \end{multline*}
     if $2^2\mid \mid a$. Thus, suppose that $a=8m$, for some integer $m\ne 0$. Then $\overline{F}(x)=x^4$ and $\gcd(\overline{F},x+1)=1$.   Since $W_1=(a-4)^2/8=2(2m-1)^2$ is squarefree, it follows that $m=1$ so that $a=8$ and $b=16$. It is straightforward to confirm that \[\GG(x)=x^8+8x^6+16x^4+8x^2+1 \quad \mbox{with} \quad \Delta(\GG)=2^{24}27^2\]  is indeed monogenic with $\Gal(\GG)\simeq$ 8T10.

 The second case to examine is
 \begin{align}\label{Eq:8T10 con2}
 \begin{split}
 W_2W_3& \quad \mbox{is a square}\\
 \mbox{while neither}\quad W_1W_3& \quad \mbox{nor} \quad W_1W_2W_3 \quad \mbox{is a square.}
 \end{split}
 \end{align} Since this case is similar to the first case, we give only a summary of the results. In this case, we found exactly two monogenic 8T10-polynomials:
 \begin{align*}
  \GG(x)=x^8-9x^6+21x^4-9x^2+1 \quad & \mbox{with} \quad \Delta(\GG)=2^85^641^2 \quad \mbox{and}\\
   \GG(x)=x^8-8x^6+16x^4-8x^2+1 \quad & \mbox{with} \quad \Delta(\GG)=2^{24}17^2.
 \end{align*}

 The third, and final, case to examine is
 \begin{align}\label{Eq:8T10 con3}
 \begin{split}
 W_1W_2W_3& \quad \mbox{is a square}\\
 \mbox{while neither}\quad W_1W_3& \quad \mbox{nor} \quad W_2W_3 \quad \mbox{is a square.}
 \end{split}
 \end{align}
If $W_1W_2W_3=1$, then using Maple to solve this equation yields four solutions, all of which require
\[b^4-24b^3+216b^2-864b+1280=(b-8)(b-4)(b^2-12b+40)\] to be a square. Using Magma ({\bf QuarticIntegralPoints}([1,-24,216,-864,1280],[4,0]);), we see that the only integral points on the curve $y^2=b^4-24b^3+216b^2-864b+1280$ are $(b,y)\in \{(4,0),(8,0)\}$. However, each of these values of $b$ produces only irrational values for $a$. Hence, $W_1W_2W_3>1$ since $W_1W_2W_3$ is a square.
 %Thus, there exists a prime $q$ such that $q^2\mid W_1W_2W_3$.

If $2\mid W_3$ then $2\mid a$, and so $4\mid W_3$. We claim that $2\nmid b$. Assume, by way of contradiction, that $2\mid b$. Note that if $b\equiv 2 \pmod{4}$, then $4\mid W_1$, contradicting the fact that $W_1$ is squarefree. Hence, $4\mid b$, $2\mid W_1$ and $2\mid W_2$. Thus, $16\mid W_1W_2W_3$, so that $M:=W_1W_2W_3/16$ is a square and
\begin{align*}
 M&=((b/2)+1-a)((b/2)+1+a)((a/2)^2+b-2)\\
  &\equiv \left\{\begin{array}{cl}
  (1)(1)(3)\equiv 3 \pmod{4} & \mbox{if $a\equiv 2 \pmod{4}$ and $b\equiv 4 \pmod{8}$,}\\
  (3)(3)(3)\equiv 3 \pmod{4} & \mbox{if $a\equiv 2 \pmod{4}$ and $b\equiv 0 \pmod{8}$,}\\
  (3)(3)(2)\equiv 2 \pmod{4} & \mbox{if $a\equiv 0 \pmod{4}$ and $b\equiv 4 \pmod{8}$,}\\
  (1)(1)(2)\equiv 2 \pmod{4} & \mbox{if $a\equiv 0 \pmod{4}$ and $b\equiv 0 \pmod{8}$,}\\
  \end{array}\right.
 \end{align*}  which is impossible in any case.   Thus, the claim that $2\nmid b$ is established, and therefore,
\[g(x)\equiv (x^2+x+1)^2 \pmod{2}.\] Note that $x^2+x+1$ is irreducible in $\F_2[x]$. Then, applying Theorem \ref{Thm:Dedekind} with $T(x):=g(x)$ and $q=2$, we can let
  \begin{equation}\label{Eq:h1h2}
  h_1(x)=h_2(x)=x^2+x+1,
  \end{equation} so that
   \begin{equation}\label{Eq:F}
   F(x)=-x\left(\left(\dfrac{a-2}{2}\right)x^2+\left(\dfrac{b-3}{2}\right)x+\left(\dfrac{a-2}{2}\right)\right).
   \end{equation} Next, we determine the possible values of $(a \mmod{4},\ b \mmod{4})$. Since $4\mid W_1W_2W_3$, then $N:=W_1W_2W_3/4$ is a square with
\begin{align}\label{Eq:N}
N&=\dfrac{(b+2-2a)(b+2+2a)(b^2-4b+8)}{4}\\
&=(b+2-2a)(b+2-2a)((a/2)^2-b+2).\nonumber
\end{align} Hence, since $2\mid a$ and $2\nmid b$, we see that $(b+2-2a)(b+2+2a)\equiv 1 \pmod{4}$ and
\[N\equiv (a/2)^2-b-2\equiv \left\{\begin{array}{cl}
-b-1 \pmod{4} & \mbox{if $a\equiv 2 \Mod{4}$}\\
-b-2 \pmod{4} & \mbox{if $a\equiv 0 \Mod{4}$.}
\end{array}\right.\]
Thus, it follows that
\begin{equation*}\label{Eq:b}
b\equiv \left\{\begin{array}{cl}
  3 \pmod{4} & \mbox{if $a\equiv 2 \Mod{4}$}\\
  1 \pmod{4} & \mbox{if $a\equiv 0 \Mod{4}$.}
\end{array}\right.
\end{equation*}
Hence, $g(x)$ and $\GG(x)$ are not monogenic by Lemma \ref{Lem:MainG2}.
%Consequently, from \eqref{Eq:F} and \eqref{Eq:b}, we have that
%\[\overline{F}(x)=\left\{\begin{array}{cl}
%  0 & \mbox{if $a\equiv 2 \Mod{4}$}\\
% x(x^2+x+1)          & \mbox{if $a\equiv 0 \Mod{4}$,}
%\end{array}\right.\]
%and we deduce from \eqref{Eq:h1h2} and Theorem \ref{Thm:Dedekind} that, in either case of the congruence class of $a$ modulo 4, $g(x)$ and $\GG(x)$ are not %monogenic.

Thus, we can assume additionally that $2\nmid a$ and $W_3$ is squarefree. We claim that $2\nmid b$. To see this, we proceed by way of contradiction and consider the two cases: $b\equiv 0 \pmod{4}$ and $b\equiv 2 \pmod{4}$. If $b\equiv 0 \pmod{4}$, then $W_1\equiv 0 \pmod{4}$, contradicting the fact that $W_1$ is squarefee. If $b\equiv 2 \pmod{4}$, then $4\mid W_1W_2W_3$ and $N$, as given in \eqref{Eq:N}, is a square with
 \[N=((b/2)+1-a)((b/2)+1+a)(a^2-4b+8)\equiv (2-a)(2+a)(1)\equiv -a^2\equiv 3 \pmod{4},\]
 which is impossible.  Hence, $2\nmid b$ and $W_1W_2W_3\not \equiv 0 \pmod{2}$.

 Note that $W_1W_2$ and $W_3$ are either both positive or both negative since $W_1W_2W_3$ is a square. If
 \[W_1W_2=(b+2)^2-4a^2<0 \quad \mbox{and}\quad W_3=a^2-4b+8<0,\] then
 \[(b+2)^2<4a^2<16b-32,\]
 which yields the contradiction
 \[(b+2)^2-16b+32=(b-6)^2<0.\]
 Hence,
 \begin{equation}\label{Eq:W1W2 W5 pos}
 W_1W_2>0\quad \mbox{and} \quad W_3>0.
 \end{equation}

 Let
 \[P=\gcd(W_1,W_3),\quad Q=\gcd(W_1,W_2)\quad \mbox{and}\quad R=\gcd(W_2,W_3).\]  Then,
  \begin{align}\label{Eq:PQR}
 \begin{split}
   \abs{W_1}&=\abs{b+2-2a}=PQ\\
   \abs{W_2}&=\abs{b+2+2a}=QR\\
  W_3&=a^2-4b+8=PR,
   \end{split}
 \end{align}
 where $PQR$ is squarefree. Thus, either $PQR=1$ or $PQR$ is the product of distinct odd primes.

 We claim that $Q\ne 1$. To see this, we assume that $Q=1$ and proceed by way of contradiction. Then, we deduce from  \eqref{Eq:W1W2 W5 pos} and \eqref{Eq:PQR} that
  \[(b+2)^2-4a^2=a^2-4b+8,\]
  so that
  \begin{equation}\label{Eq:a2}
  a^2=\dfrac{b^2+8b-4}{5}.
  \end{equation} Thus,
  \begin{equation}\label{Eq:W1W2SF}
  W_1W_2=(b+2)^2-4a^2=(b+2)^2-4\left(\dfrac{b^2+8b-4}{5}\right)=\dfrac{(b-6)^2}{5}.
  \end{equation} Since $Q=1$, it follows that $W_1W_2$ is squarefree. Hence, from \eqref{Eq:W1W2SF}, we conclude that $b\in \{1,11\}$. If $b=1$, then $a=\pm 1$ from \eqref{Eq:a2}. However, if $a=1$, we arrive at the contradiction that $W_1=1$, while if $a=-1$, we get the contradiction that $W_2=1$. If $b=11$, then $a^2=41$ from \eqref{Eq:a2}. Therefore, the claim is established, and $Q\ge 3$, since $Q$ is odd.

 Note that if $P=R\ne 1$, then $W_3=R^2>1$, which contradicts the fact that $W_3$ is squarefree.
 If $P=R=1$, then $W_1W_2=\abs{W_1W_2}=Q^2$, which contradicts the fact that $W_1W_2$ is not a square. Therefore, $P\ne R$. Similar arguments show that $P\ne Q$ and $Q\ne R$.

We proceed by providing details in the situation when $W_1>0$ and $W_2>0$. We omit details when $W_1<0$ and $W_2<0$ since the arguments are similar, and no new solutions arise. Invoking Maple to solve the system \eqref{Eq:PQR}, we get that
\begin{equation}\label{Eq:Main}
P^2Q^2-2PQ^2R+Q^2R^2-32PQ-32QR-16PR+256=0.
\end{equation}
It follows from \eqref{Eq:Main} that
\begin{equation}\label{Eq:Div}
P\mid (QR-16),\quad Q\mid (PR-16)\quad \mbox{and} \quad R\mid (PQ-16).
\end{equation}
  Thus, since $PQR$ is squarefree, we deduce from \eqref{Eq:Div} that $PQR$ divides
  \begin{align*}
  Z:&=\dfrac{(QR-16)(PR-16)(PQ-16)-PQR(PQR-16P-16Q-16R)}{256}\\
  &=PQ+QR+PR-16.
  \end{align*}
It is easy to see that $Z\ge 7$ since $P$, $Q$ and $R$ are distinct odd positive integers. Hence, since $PQR$ divides $Z$, we have that $H:=PQR-Z\le 0$.
On the other hand, using Maple, we see that the minimum value of $H$, subject to the constraints $\{P\ge 3,Q\ge 3,R\ge 3\}$, is 16. Thus, we deduce that $P=1$ or $R=1$. Letting $P=1$ in \eqref{Eq:Main}, and solving for $Q$ yields
\begin{equation}\label{Q}
Q=\dfrac{4(4R+4\pm \sqrt{R^3-2R^2+65R})}{(R-1)^2},
\end{equation}
while solving for $R$ produces
\begin{equation}\label{R}
R=\dfrac{Q^2+16Q+8\pm 4\sqrt{4Q^3+Q^2+16Q+4}}{Q^2}.
\end{equation} For $Q$ to be a viable solution, we conclude from \eqref{Q} that
\begin{equation}\label{Eq:EQ}
y^2=R^3-2R^2+65R,
\end{equation}
for some integer $y$. Using Sage to find all integral points (with $y\ge 0$) on the elliptic curve \eqref{Eq:EQ} we get
\[(R,y)\in \{(0,0),(1,8),(5,20),(13,52),(16,68),(45,300),(65,520),(1573,62348)\}.\] In \eqref{Q}, we cannot have $R=1$, and since $R\ge 1$ is odd and squarefree, we have that
$R\in \{5,13,65\}$. Plugging these values back into \eqref{Q} reveals only the two valid solution triples $(P,Q,R)\in \{(1,11,5), (1,3,13)\}$. Although the triple $(1,1,5)$ arises here, it is not legitimate since we know that $P\ne Q$. Because \eqref{Eq:Main} is symmetric in $P$ and $R$, we can also derive the two additional solutions $(5,11,1)$ and $(13,3,1)$. Using the same approach on \eqref{R} yields no additional solutions. Also, the substitution of $R=1$ into \eqref{Eq:Main} produces no additional solutions as well. In summary, the only solutions are given by
\begin{equation}\label{Sols}
(P,Q,R)\in \{(1,11,5), (1,3,13),(5,11,1), (13,3,1)\}.
\end{equation}
The corresponding values of $(a,b)$ can be found by using the values of $P$, $Q$ and $R$ in \eqref{Sols}, and solving the resulting systems in \eqref{Eq:PQR}. These pairs are
\[(a,b)\in \{(11,31),(9,19),(-11,31),(-9,19)\},\] and the corresponding polynomials are
\begin{align*}\label{Gi}
\begin{split}
\GG(x)&=x^8+11x^6+31x^4+11x^2+1 \quad \mbox{with $\Delta(\GG)=2^85^611^4$,}\\
\GG(x)&=x^8+9x^6+19x^4+9x^2+1\quad \mbox{with $\Delta(\GG)=2^83^413^6$,}\\
\GG(x)&=x^8-11x^6+31x^4-11x^2+1\quad \mbox{with $\Delta(\GG)=2^85^611^4$,}\\
\GG(x)&=x^8-9x^6+31x^4-9x^2+1\quad \mbox{with $\Delta(\GG)=2^83^413^6$.}
\end{split}
\end{align*}

In summary, we have shown that there exist exactly seven distinct monogenic polynomials $\GG(x)$ with $\Gal(\GG)\simeq$ 8T10 (see Table \ref{T:2}), and Magma confirms that these polynomials generate distinct octic fields.

%It is easy to verify that $\GG_i(x)$ is monogenic and that $\Gal(\GG_i)\simeq$ 8T10 for $i\in \{1,2,3,4\}$. From the discriminants in \eqref{Gi}, we see that $\GG_1(x)$ and $\GG_2(x)$, as well as $\GG_3(x)$ and $\GG_4(x)$, generate distinct octic fields. Also, $\GG_1(x)$ and $\GG_3(x)$, as well as $\GG_2(x)$ and $\GG_4(x)$, generate distinct octic fields since the octic fields generated by $\GG_1(x)$ and $\GG_2(x)$ are imaginary, while the octic fields generated by $\GG_3(x)$ and $\GG_4(x)$ are real.

\subsection*{X=18} Let
\[\GG_t(x):=x^8+3x^6+(2t+1)x^4+3x^2+1.\] Then
$\Delta(\GG_t)=256(2t+9)^2(2t-3)^2(8t-13)^4$. Let $G(t):=(2t+9)(2t-3)(8t-13)$. Since $G(1)\equiv 3 \pmod{4}$, we deduce from Lemma \ref{Lem:ObstructionCheck} that $G(t)$ has no local obstructions. Hence, from Corollary \ref{Cor:Squarefree}, there exist infinitely many primes $p$ such that $G(p)$ is squarefree. Let $p\ge 3$ be such a prime. %We claim that $\Gal(\GG_p)\simeq$ 8T18 and that $\GG_p(x)$ is monogenic.
Observe then for $\GG_p(x)$ we have
\begin{equation}\label{Eq:Wi8T18}
  W_1=2p-3, \quad W_2=2p+9 \quad \mbox{and} \quad W_3=-8p+13.
\end{equation}
Since $G(p)$ is squarefree, it follows that $\Gal(\GG)\simeq$ 8T18 from Theorem \ref{Thm:AP1}, and that $\GG(x)$ is monogenic from Lemma \ref{Lem:BAMS}.

  Suppose that, for some primes $3\le p<q$, with $G(p)$ and $G(q)$ both squarefree, we have that $\GG_p(x)$ and $\GG_q(x)$ generate the same octic field. Then, it must be that $\Delta(\GG_p)=\Delta(\GG_q)$, which implies that
\[256(2p+9)^2(2p-3)^2(8p-13)^4=256(2q+9)^2(2q-3)^2(8q-13)^4,\] contradicting the fact that $p<q$.
Thus, $\{\GG_p(x): G(p) \mbox{ is squarefree}\}$
is an infinite family of monogenic even octic 8T18-polynomials, which completes the proof of the theorem.
%\begin{rem}
%  Note that the collection of monogenic polynomials given here is not exhaustive.
%\end{rem}
\end{proof}

We have summarized the results of Theorem \ref{Thm:Main} below in Tables \ref{T:2} and \ref{T:3}.\\

  \begin{table}[h]
 \begin{center}
\begin{tabular}{c|ccc}
 X & Distinct Monogenic 8TX-Trinomials $\FF(x)$& \#\\ [2pt]  \hline \\[-8pt]
 2& $x^8+1$&1\\
 3& $x^8-x^4+1$ & 1\\
 4& $x^8+3x^4+1$ & 1\\
 6& $x^8+2$ & 1\\
 8& $x^8-2x^4-1,\quad$ $x^8-2$ & 2\\
 9& $x^8+(4p+3)x^4+1$ & $\infty$\\
 &\mbox{where $p$ is prime with $(4p+1)(4p+5)$ squarefree} & \\
 11& none & 0\\
 15& $x^8-ax^4-1$ & $\infty$\\
 & \mbox{where $a>0$, $4\nmid a$ with $(a^2+4)/\gcd(a^2,4)$ squarefree} &\\
 16& $x^8-4x^4+2$, \ $x^8+4x^4+2$, \ $x^8-5x^4+5$ & 3\\
 17& $x^8+2tx^4+t^2+1$ & $\infty$\\
 & \mbox{where $t>0$ with $t^2+1$ squarefree} &\\
 22& none & 0\\
 26& $x^8+px^4+3$ & $\infty$\\
 & \mbox{where $p\ge 5$ is prime with $p^2-12$ squarefree} & \\
 \end{tabular}
\end{center}
\caption{Monogenic 8TX-even-trinomials with X$\in X_{\FF}$}
 \label{T:2}
\end{table}

  \begin{table}[h]
 \begin{center}
\begin{tabular}{c|ccc}
 X & Distinct Monogenic 8TX-Polynomials $\GG(x)$& \# \\ [2pt]  \hline \\[-8pt]
 2& $x^8-x^6+x^4-x^2+1$ & 1\\
 3& none & 0\\
 4& none & 0\\
 9& $x^8+(4p+3)x^6+(8p+5)x^4+(4p+3)x^2+1$&$\infty$ \\
 & \mbox{where $p$ is prime with $(16p+13)(4p+1)(4p-3)$ squarefree} & \\
 10& $x^8+8x^6+16x^4+8x^2$, & 7\\
 & $x^8-9x^6+21x^4-9x^2+1$, \quad  $x^8-8x^6+16x^4-8x^2$ & \\
 & $x^8+11x^6+31x^4+11x^2+1,\quad$ $x^8+9x^6+19x^4+9x^2+1$ & \\
 & $x^8-11x^6+31x^4-11x^2+1,\quad$ $x^8-9x^6+19x^4-9x^2+1$ &\\
 18& $x^8+3x^6+(2p+1)x^4+3x^2+1$ & $\infty$ \\
 & \mbox{where $p$ is prime with $(2p+9)(2p-3)(8p-13)$ squarefree} & \\
 \end{tabular}
\end{center}
\caption{Monogenic 8TX-even-reciprocal polynomials with X$\in X_{\GG}$}
 \label{T:3}
\end{table}

% For alignments use AmS-LaTeX constructions not \eqnarray.

%% - theorems and proofs
%\begin{thm}[optional text]
% The optional material will be typeset as part of the theorem heading
%\end{thm}

%\begin{proof}[Optional proof heading]
% the proof
%\end{proof}
% An end-of-proof symbol (open box) will be typeset at the
% end of the proof.

\end{document}